\documentclass[11pt,a4paper]{article}

\usepackage{amsfonts}
\usepackage[latin9]{inputenc}
\usepackage{amsmath}
\usepackage{amssymb}
\usepackage{breqn}
\usepackage{url}
\usepackage{graphicx}
\usepackage[pdfstartview=FitH]{hyperref}
\usepackage{amsthm}
\usepackage{hyperref}
\usepackage{url}
\usepackage{enumitem}
\usepackage{authblk}
\usepackage{bbm}
\usepackage{comment}
\usepackage{cleveref}
\usepackage{fullpage}
\usepackage{microtype}

\newcommand{\Ker}{\operatorname{Ker}}

\newcommand{\Aut}{\operatorname{Aut}}

\newcommand{\supp}{\operatorname{supp}}
\newcommand{\cosys}{\operatorname{cosys}}
\newcommand{\cobound}{\operatorname{cobound}}
\newcommand{\EL}{\operatorname{EL}}
\newcommand{\SL}{\operatorname{SL}}
\newcommand{\St}{\operatorname{St}}
\newcommand{\ord}{\operatorname{ord}}
\newcommand{\Unip}{\operatorname{Unip}}
\newcommand{\Sym}{\operatorname{Sym}}

\makeatletter

\begin{document}
\newtheorem{theorem}{Theorem}[section]
\newtheorem{lemma}[theorem]{Lemma}
\newtheorem{definition}[theorem]{Definition}
\newtheorem{claim}[theorem]{Claim}
\newtheorem{example}[theorem]{Example}
\newtheorem{remark}[theorem]{Remark}
\newtheorem{proposition}[theorem]{Proposition}
\newtheorem{corollary}[theorem]{Corollary}
\newtheorem{observation}[theorem]{Observation}

\author{Tali Kaufman}
\affil{Department of Computer Science, Bar-Ilan University, kaufmant@mit.edu, }

\author{Izhar Oppenheim}
\affil{Department of Mathematics, Ben-Gurion University of the Negev, Be'er Sheva 84105, Israel, izharo@bgu.ac.il}

\author{Shmuel Weinberger}
\affil{Department of Mathematics, University of Chicago,  shmuel@math.uchicago.edu}

\title{Coboundary expansion of coset complexes}

\maketitle

\begin{abstract}
Coboundary expansion is a high dimensional generalization of the Cheeger constant to simplicial complexes.  Originally,  this notion was motivated by the fact that it implies topological expansion, but nowadays a significant part of the motivation stems from its deep connection to problems in theoretical computer science such as agreement expansion in the low soundness regime.  In this paper,  we prove coboundary expansion with non-Abelian coefficients for the coset complex construction of Kaufman and Oppenheim.  Our proof uses a novel global argument, as opposed to the local-to-global arguments that are used to prove cosystolic expansion.
\end{abstract}

\section{Introduction}

In this paper we show that the coset complexes that were introduced by \cite{KO-construction} (or in fact a variant of those) are coboundary expanders over the symmetric group.  In addition to the importance of the coboundary expansion result per se,  our work suggests that the recent result of \cite{BMV} that construct PCPs over some variants of the Ramanujan complexes, could potentially be implemented also over the coset complexes to yield efficient PCPs over the coset complexes, i.e.,  efficient PCPs over elementary high dimensional expanders.

\subsection{Our contribution}
 In this paper, we will show that for every finite group $\Lambda$,  a slight variation of the Kaufman-Oppenheim complexes give examples of bounded degree $1$-coboundary expanders over $\Lambda$.

More explicitly,  for $n \geq 3$ and $p$ prime,   we construct an infinite family of $n$-dimensional coset complexes $\lbrace X_{n,p}^{(s)} \rbrace_{s > 3n}$ modelled over the family of groups $\lbrace \SL_{n+1} (\mathbb{F}_p [t] / \langle t^s \rangle) \rbrace_{s >3n}$ using a slight variation of the construction in \cite{KO-construction}.   For this construction,  our main result in this paper is the following:
\begin{theorem}[Main Theorem - informal,  see formal Theorem \ref{vanishing of coho for X_n,p^s thm}]
Let $n \geq 3$ and $\Lambda$ a finite group.   For every prime $p$ that is large enough with respect to $n$ and $\vert \Lambda \vert$,  the family $\lbrace X_{n,p}^{(s)} \rbrace_{s > 3n}$ has uniformly bounded degree and (uniformly) $1$-coboundary expansion over $\Lambda$.
\end{theorem}

The main difficulty in the proof of our main Theorem is showing vanishing of cohomology with respect to a non-Abelian group $\Lambda$ - this is due to the fact that cosystolic expansion for such groups is known via a combination of previous work. Thus, once we prove vanishing of cohomology the previous known cosystolic result for coset complexes could be upgraded to a coboundary expansion result.  We note that vanishing of cohomology with finite coefficients is inherently a global property and cannot be attained via local to global considerations. This is the main challenge we are facing, as most developed techniques within high dimensional expansion are of local to global nature, but a local to global argument can not imply vanishing of cohomology when working with finite coefficients.  We note that our proof of this result is fairly elementary and self-contained - we only use some external results regarding the presentation of $\SL_{n+1} (\mathbb{F}_p [t])$    and some of its subgroups in terms of generators and relations.  Notably,  unlike similar results obtained in \cite{BLM, DDL} (for a completely different family of complexes),  we do not use any sophisticated tools from the theory of algebraic groups nor do we need the strong results of the congruence subgroup property and the strong approximation theorem.

\subsection{Significance of coboundary expansion over the symmetric group within CS}
In the following,  we discuss the notions of coboundary expansion, cosystolic expansion and agreement expansion and their relation to PCP construction.

\textbf{Coboundary expansion. } Coboundary expansion  is a topological notion of high dimensional expansion that was introduced by Gromov \cite{Grom} and independently by Linial and Meshulam \cite{LM}.  Gromov has studied this notion since he has shown that coboundary expansion over $\mathbb{F}_2$ of a simplicial complex implies the topological overlapping property of the complex.  Gromov knew how to prove this topological expansion property for complexes with small diameter, but his main interest was to obtain it for complexes with unbounded diameter,  such as the Ramanujan complexes defined in \cite{LSV2,  LSV1} \footnote{When we refer to Ramanujan complexes below,  any quotient of an affine $\widetilde{A}_n$ building with sufficiently large injectivity radius can be used.} since he wanted to show the existence of bounded degree complexes with the topological overlapping property. This question of Gromov was resolved by \cite{KKL,EK} using a relaxed notion to coboundary expansion called cosystolic expansion.

\textbf{Cosystolic expansion.} Cosystolic expansion is the requirement that an approximate cocycle (i.e.,  a cochain whose coboundary is small) is close to a genuine cocycle. Coboundary expansion is equivalent to cosystolic expansion plus additional requirement of vanishing of the relevant cohomologies.  Importantly, cosystolic expansion is a global testability question that could be deduced by local to global means (see \cite{KKL,EK} that have shown cosystolic expansion with $\mathbb{F}_2$ coefficients). However,  in contrast to cosystolic expansion, coboundary expansion is a global property that can not be deduced by local to global means. The techniques to get cosystolic expansion with $\mathbb{F}_2$ coefficients by local to global means were also pivotal in obtaining good LTCS and good qLDPCs \cite{EKZ,  DELLM,  PK}.

\textbf{Coset complexes.} The Ramanujan complexes were shown to be bounded degree cosystolic expanders,  but their construction was rather involved.  This raised the question of  the existence of bounded degree cosystolic expanders with an elementary construction.  The first construction of bounded degree spectral high dimensional expanders was given in  \cite{KO-construction} using the idea of coset complexes.  This idea was generalized in \cite{KO-construction2, O'DP,  KacM-complexes} to other similar constructions of bounded degree spectral high dimensional expanders.  It was later shown in \cite{KO-cobound},  that the complexes constructed in \cite{KO-construction} are also give rise bounded degree cosystolic expanders over $\mathbb{F}_2$ in dimension 2 (it is conjectured that this result can be generalized to higher dimensions).  These construction remain one of the main sources of examples for HDX's (e.g.,  see \cite{NewCodes}).

\textbf{Agreement expansion in the high and low soundness regimes. } In another line of work within CS \cite{DK} have studied the notion of agreement expansion and have shown that high dimensional \textbf{spectral} expanders are agreement expanders in the high soundness regime; Agreement expansion in the high soundness regime asks roughly for the following: Given a binary assignment to the vertices of each k face in a complex of dimension $n >k$ such that almost all k-assignments are agreeing on their intersections, is it the case that there exists a global binary function on all the vertices of the complex whose restriction to the different k-faces agrees with most assignments to the k-faces. This question is related to PCPs and is called \textit{agreement expansion in the high soundness regime}.  However for PCPs constructions one needs a stronger requirement called \textit{agreement expansion in the low soundless regime}. Agreement expansion in the low soundless regime roughly requires that if 1 percent of the intersection of the k-faces are agreeing then there exists a global function that is consistent with a constant fraction of the k-sets.  In \cite{DK} it was shown that spectral high dimensional expansion is sufficient for agreement expansion in the high soundness regime and they have conjectured that it should imply also the low soundness case.  So the state of things until recently has been that agreement expansion in the high soundness case is implied by spectral high dimensional expansion and robustness of codes and topological overlapping property are implied by cosystolic expansion (i.e. by topological expansion) or similar variants (see also \cite{FK}).

\textbf{List agreement expansion and the bridge between spectral and topological high dimensional expansion. } The work of \cite{GK} has recently defined a stronger version of the agreement expansion that also needs topological expansion in additional to spectral high dimensional expansion.  They have defined the notion of list agreement expansion ; in which the agreement expansion question in the high soundness regime is changed so that the input on each k-face is
a list of $l$ different assignments to the k-face (and not only one) and the question is the following : if a typical pair of k-faces that intersect are agreeing on their whole list then there exist $l$-global functions that are consistent with most assignments of $l$-lists on $k$-faces.  In \cite{GK} it was also shown that if the complex is both a spectral high dimensional expander and a coboundary expander over the symmetric group then it has the list-agreement expansion property in the high soundness regime.  Then the following question arises: \textit{what is the relation between agreement expansion in the low soundness regime,  that is useful towards PCP construction and the list-agreement expansion in the high soundness regime?}

\textbf{Reduction from agreement with low soundness to list agreement with high soundness.} In two independent works \cite{BM} and \cite{DD-swap} have shown that there is a reduction from agreement expansion in the low soundness regime to list agreement expansion in the high soundness.  Namely one can solve low soundness agreement by translating it to high soundness agreement on lists (with additional requirement that the coboundary expansion constant in links is independent of the complex dimension). As was explained, high soundness agreement on lists can be solved by coboundary expansion of the complex and by spectral expansion of it.

\textbf{The Path towards PCP construction. } After those works that have shown a reduction from agreement expansion in the low soundness to list agreement expansion in the high soundness regime, the main challenge towards obtaining bounded degree complexes that supports agreement expansion in the low soundness regime has been in coming up with a bounded degree complexes that are coboundary expanders over the symmetric group. \cite{BLM, DDL} have shown the existence of such complexes based on an earlier work of \cite{CL}. Thus, these works have shown the existence of bounded degree complexes that support agreement expansion in the low soundness regime. All these efforts have culminated in the \cite{BMV} that showed that a complex which supports agreement expansion in the low soundness regime gives rise to an efficient PCP by routing through it existing local PCP.

\textbf{On the implications of our work. } Based on the previous discussion,  the main global challenge to get a PCP from a complex with expanding links is in proving the global property that this complex is a coboundary expander over the symmetric group. The main difficulty is that coboundary expansion of the complex is a property that can not be implied by local to global means that are abundant in the high dimensional expansion literature.

We see the significance of our results in three key aspects: First, proving coboundary expansion of the coset complex over the symmetric group.  Second, solving the main global challenge towards constructing PCPs based on the coset complex (although the local problem of providing coboundary expansion constants of the links that are independent of the degree is still open). Third,  we provide a novel global technique for proving vanishing of the first cohomology over the symmetric group (in contrast to other techniques for establishing spectral expansion or cosystolic expansion that are of local nature).  We believe that our work can be generalized to the other constructions of spectral expanders using coset complexes in \cite{O'DP}.

\subsection{Comparison to similar results}

It is intriguing to compare our proof of the vanishing of cohomology with recent results in \cite{BLM, DDL}. Given a finite group $\Lambda$ both our work and \cite{BLM, DDL} construct a family of high-dimensional expanders with trivial 1-cohomology relative to $\Lambda$-coefficients, following a similar outline. The approach begins with an "algebraic" simplicial complex that has trivial 1-cohomology with respect to $\Lambda$,  and then passes to a family of high-dimensional expanders by taking quotients via congruence subgroups, which also have trivial 1-cohomology with $\Lambda$-coefficients.

Despite this shared framework, the methods and challenges in the proofs differ significantly. In \cite{BLM, DDL}, the initial complex is a contractible symplectic
$\widetilde{C}_n$-building, which has trivial cohomology for any group. The difficulty lies in proving that the congruence subgroups exhibit trivial 1-cohomology with respect to $\Lambda$,  requiring deep results from group theory, such as the congruence subgroup property and the strong approximation theorem.

In contrast, our work begins with a coset complex over $\SL_{n+1} (\mathbb{F}_p [t])$.  For $n \geq 3$,  the explicit description of the congruence subgroups, which follows from the fact that $\SL_{n+1} (\mathbb{F}_p [t])$ has the same presentation as  a Steinberg group, makes the vanishing of 1-cohomology for congruence subgroups with $\Lambda$-coefficients almost trivial when $p > \vert \Lambda \vert$.  This is a straightforward argument requiring no deep theoretical results. Most of our effort focuses on proving that the coset complex over $\SL_{n+1} (\mathbb{F}_p [t])$ has trivial 1-cohomology with respect to $\Lambda$ (unlike the work in \cite{BLM, DDL}, the coset complex in our case is not simply connected).

\subsection{Proof overview}

In order to show that a family of complexes $\lbrace X^{(s)} \rbrace$ has $1$-coboundary expansion over a group $\Lambda$,  one needs to show two things:  First,  that the family has $1$-cosystolic expansion over $\Lambda$.  Second,  that for every $s$,  $H^1 (X^{(s)}, \Lambda) =0$.

In \cite{KO-cobound} it was shown that the Kaufman-Oppenheim complexes (with large enough prime $p$) have $1$-cosystolic expansion over $\mathbb{F}_2$.  The new results of Dikstein and Dinur in \cite{DD-cosys, DD-swap} show that the same proof given in \cite{KO-cobound} actually show that the Kaufman-Oppenheim complexes have $1$-cosystolic expansion over any group $\Lambda$ (see a more detailed explanation of this point in  \Cref{The Kaufman-Oppenheim coset complexes subsec} below).

Thus we are left to show a vanishing of cohomology result for the coset complexes in our construction.  This is done via the following steps:
\paragraph*{\textbf{Vanishing of cohomology for a quotient.}} Let $X$ be a simplicial complex and $N$ a group acting simplicially on $X$.  The quotient complex, denoted $N \backslash X$,  is defined via identifying all the vertices of $X$ that are in the same orbit of $N$.  Under some mild conditions on $X$ and on the action of $N$,  the following result holds for every group $\Lambda$:  If both $H^1 (X, \Lambda) =0$ and $H^1 (N, \Lambda)  =0$, then $H^1 (N \backslash X,  \Lambda) =0$ (see exact formulation in Theorem \ref{N mod X thm} below).

\paragraph*{\textbf{Passing to the coset complex of $ \SL_{n+1} (\mathbb{F}_p [t] )$.}} It was already observed in \cite{KO-construction2} that the Kaufman-Oppenheim coset complexes over the family of groups $\lbrace \SL_{n+1} (\mathbb{F}_p [t] / \langle t^s \rangle) \rbrace_{s >3n}$ are actually all quotients of a coset complex modelled over $\SL_{n+1} (\mathbb{F}_p [t])$.  Namely,  for any $n,p$, there are coset complexes $X_{n,p}$  modelled over $\SL_{n+1} (\mathbb{F}_p [t])$ and normal subgroups  $\Gamma_{n,p}^s  \triangleleft \SL_{n+1} (\mathbb{F}_p [t])$ such that for any $s >3n$,  $X_{n,p}^{(s)}  = \Gamma_{n,p}^s  \backslash X_{n,p}$.  We will see that if $\Lambda$ is a group that has no non-trivial element of order $p$,  then $H^1 (\Gamma_{n,p}^s,  \Lambda) = 0$ for every $s$.  In particular,  if $\Lambda$ is finite and $p > \vert \Lambda \vert $,  then $H^1 (\Gamma_{n,p}^s,  \Lambda) = 0$ for every $s$.  By the discussion above,  if we can show that for such $p > \vert \Lambda \vert$,  it also holds that $H^1 (X_{n,p},  \Lambda) =0$, then it will follow that $H^1 (X_{n,p}^{(s)}, \Lambda) =0$ for every $s >3n$.

\paragraph*{\textbf{Vanishing of cohomology for $X_{n,p}$.}} Fix $n \geq 3$ and $p$ an odd prime.  To avoid cumbersome notation,  we will denote $X = X_{n,p}$.  By the previous paragraph,  we need to show that for any finite group $\Lambda$,  if $p > \vert \Lambda \vert$,  then $H^1 (X, \Lambda)=0$.  Had we known that $X$ is simply connected,  then we would be done,  since for every simply connected complexes, the $1$-cohomology vanishes with respect to any group $\Lambda$.  However,  $X$ is a coset complex and there is a characterization of simple connectedness for coset complexes (see \cite{AbelsH} and further discussion below) that $X$ does not meet.  Fortunately,  the universal cover of $X$,  denoted $\widetilde{X}$,  has an explicit description as a coset complex.  Namely,  there is an explicit abstract group $\widetilde{\Gamma}$ and a simply connected coset complex $\widetilde{X}$ modelled over $\widetilde{\Gamma}$ such that $\widetilde{\Gamma} \rightarrow \SL_{n+1} (\mathbb{F}_p [t])$ is a surjective homomorphism that induces a covering map $\widetilde{X} \rightarrow X$.  The group $\widetilde{\Gamma}$ is given in terms of generators and relations such that the generators can be identified with a generating set of $\SL_{n+1} (\mathbb{F}_p [t])$ and the set of relations of $\widetilde{\Gamma}$ is a partial set of the relations in a finite presentation of $\SL_{n+1} (\mathbb{F}_p [t])$.  We will refer to the relations in the presentation of $\SL_{n+1} (\mathbb{F}_p [t])$ that do not appear in the presentation of $\widetilde{\Gamma}$ as the ``missing relations''.   Adding relations to $\widetilde{\Gamma}$ and considering the coset complex with the added relations is equivalent to passing to a quotient of $\widetilde{X}$ by the normal subgroup generated by the added relations.  The crux of our proof is that we can add the missing relations to   $\widetilde{\Gamma}$ in an iterative process,  such that the normal group that we divide by has no $\Lambda$ $1$-cohomology.  Thus,  we start with $H^1 (\widetilde{X},  \Lambda) =0$ (since $\widetilde{X}$ is simply connected) and pass in a sequence of quotients $X_0 = \widetilde{X},  X_1 = N_0 \backslash X_0,..., X_m = N_{m-1} \backslash X_{m-1}$ such that $X_m = X$ and in each step $H^1 (N_i,  \Lambda) =0$ and thus $H^1 (X_{i+1},  \Lambda) =0$ by the discussion above.  This result builds on identifying the relations in $\SL_{n+1} (\mathbb{F}_p [t])$ as pairs of roots in the root system $A_n$ and establishing a combinatorial propagation result in the chambers of $A_n$ (passing from the relations in $\widetilde{\Gamma}$ to the missing relations).

\section{Background}

\subsection{Simplicial complexes}

An $n$-dimensional simplicial complex $X$ is a hypergraph whose hyperedges of maximal size are of size $n+1$, and which is closed under containment. Namely, for every hyperedge $\tau$ (called a face) in $X$, and every $\eta\subset\tau$, it must be that $\eta$ is also in $X$. In particular, $\emptyset \in X$. For example, a graph is a $1$-dimensional simplicial complex. Let $X$ be a simplicial complex, we fix the following terminology/notation:
\begin{enumerate}
\item The simplicial complex $X$ can be thought of as a set system of $V (X)$,  where $V (X)$ is called the \textit{vertex set} of $X$ and every element of $V(X)$ is called a vertex of $X$.  By abuse on notation,  below,  we will use $V (X)$ and $X (0)$ interchangeably although they are formally different (formally,  $X (0) = \lbrace \lbrace v \rbrace : v \in V (X) \rbrace$.
\item The simplicial complex $X$ is called {\em pure $n$-dimensional} if every face in $X$ is contained in some face of size $n+1$.
\item The set of all $k$-faces (or $k$-simplices) of $X$ is denoted $X(k)$, and we will be using the convention in which $X(-1) = \{\emptyset\}$.
\item For $0 \leq k \leq n$, the $k$-skeleton of $X$ is the $k$-dimensional simplicial complex $X(0) \cup X(1) \cup ... \cup X(k)$. In particular, the $1$-skeleton of $X$ is the graph whose vertex set is $X(0)$ and whose edge set is $X(1)$.
\item For a simplex $\tau \in X$, the link of $\tau$, denoted $X_\tau$ is the complex
$$\lbrace \eta \in X : \tau \cup \eta \in X, \tau \cap \eta = \emptyset \rbrace.$$
We note that if $\tau \in X(k)$ and $X$ is pure $n$-dimensional, then $X_\tau$ is pure $(n-k-1)$-dimensional.
\item Given a pure $n$-dimensional simplicial complex,  the \textit{weight function} $ w : \bigcup_{k=-1}^n X(k) \rightarrow \mathbb{R}_+$ is defined to be
$$\forall \tau \in X(k), w(\tau) = \frac{\vert \lbrace \sigma \in X(n) : \tau \subseteq \sigma \rbrace \vert}{{n+1 \choose k+1} \vert X(n) \vert}.$$
\item A family of pure $n$-dimensional simplicial complexes  $\lbrace X_{(s)} \rbrace_{s \in \mathbb{N}}$ is said to have bounded degree if there is a constant $L>0$ such that for every $s \in \mathbb{N}$ and every vertex $v$ in $X^{(s)}$, $v$ is contained in at most $L$ $n$-dimensional simplices of $X^{(s)}$.
\item An $n$-dimensional simplicial complex $X$ is called {\em partite} (or colorable) if its set of vertices can be partitioned into $n+1$ disjoints sets $\mathcal{S}_0,...,\mathcal{S}_n$ that are called the \textit{sides of $X$} such that for every $0 \leq i \leq n$ and every $v, u \in \mathcal{S}_i$,  $ \lbrace u,v \rbrace \notin X(1)$.  Equivalently,  the vertices of $X$ can be colored in $n+1$ colors such that vertices of the same color are not connected by an edge.  If $X$ is pure $n$-dimensional,  this is equivalent to the condition that for every $\sigma \in X (n)$ and every $0 \leq i \leq n$,  $\vert \sigma \cap S_i \vert =1$ (or equivalently, that the vertices of $X$ can be colored in $n+1$ colors such that every $n$-simplex has vertices of all colors).
\item Given a constant $0 \leq \lambda <1$,  a pure $n$-dimensional simplicial complex $X$ is said to have $\lambda$-one-sided local spectral expansion,  if $X$ and all its links are connected and for every $-1 \leq k \leq n-2$ and every $\tau \in X(k)$,  the random walk on the $1$-skeleton on $X_\tau$ has a spectral gap $\geq 1-\lambda$ (the random walk is weighted according to the weight function $w$ - for an exact definition see \cite{KO-construction}).
\end{enumerate}

Let $X,  Y$ be two $n$-dimensional simplical complexes and $f: V (X) \rightarrow V (Y)$ a function.
\begin{enumerate}
\item The map $f$ is called \textit{simplicial} if for every $\sigma \in X$ it holds that $f (\sigma) \in Y$.  In that case, we also denote $f: X \rightarrow Y$.  Note that $f (\sigma)$ need not be a simplex in the same dimension as the dimension of $\sigma$.
\item The map $f$ is called \textit{rigid} if it is simplicial and for every $\sigma \in X$,  $f (\sigma)$ is of the same dimension as $\sigma$,  i.e., $\vert f (\sigma) \vert = \vert \sigma \vert$.
\item The map $f$ is called an \textit{isomorphism} if it is injective,  simplicial and $f^{-1}$ is also simplicial (note that in particular it must hold that $f,  f^{-1}$ are rigid).
\item If $X$ and $Y$ are partite with sides $\mathcal{S}_0 (X),...,\mathcal{S}_n (X)$ and $\mathcal{S}_0 (Y),...,\mathcal{S}_n (Y)$ correspondingly.  The function $f$ is called \textit{color preserving} if it is simplicial and for every $0 \leq i \leq n$,  $f (\mathcal{S}_i (X)) \subseteq \mathcal{S}_i (Y)$.
\end{enumerate}

\begin{observation}
For $X,Y,  f$ as above,  if $X$ and $Y$ are partite and $f$ is color preserving,  then it is rigid.
\end{observation}

Given a simplicial complex $X$,  we denote $\Aut (X)$ to be the group of all automorphisms of $X$,  i.e., the group of all the functions $f: X \rightarrow X$ that are isomorphisms with the binary action of composition.  For a group $G$,  we say that there is a \textit{group action of $G$ on $X$ }(or that \textit{$G$ acts on $X$}) if there is a homomorphism $\pi : G \rightarrow \Aut (X)$.  We say that the action of $G$ on $X$ is \textit{color preserving} if for every $g \in G$,  $\pi (g)$ is color preserving.   Below,  the homomorphism $\pi$ will usually be implicit and we will simply treat each $g \in G$ as an isomorphism $g : X \rightarrow X$ and write $g.\sigma$ to denote the action of an element $g \in G$ on $\sigma \in X$.

Given a group $G$ acting on $X$,  the \textit{quotient} the action of $X$ is the simplicial complex denoted $G \backslash X$ that is defined explicitly as follows: Define an equivalence relation $\sim$ on $V (X)$ by $v \sim g.v$ for every $g \in G$, i.e., two vertices are equivalent if they are in the same orbit of the action .  The vertex set $V (G \backslash X)$ is the set of equivalence classes $\lbrace [v] : v \in V (X) \rbrace$ under this relation (equivalently, $[v]$ is the orbit of $v$ under the action of $G$).  For $[v_0],...,[v_k] \in V(G \backslash X)$ it holds that $\lbrace [ v_0 ],..., [v_k ] \rbrace \in (G \backslash X) (k)$ if and only if there are $u_i \in [v_i]$ for every $0 \leq i \leq k$ such that $\lbrace u_0,...,u_k \rbrace \in X(k)$.  The \textit{projection} $p : X \rightarrow G \backslash X$ is defined to be the map induced by $p (v) = [v]$ for every $v \in V (X)$.  We note that $p$ is always surjective and simplicial, but needed not be rigid.

\begin{observation}
\label{color presev action is rigid obs}
For $X$ and $G$ as above,  if $X$ is $n$-dimensional and partite and the action of $G$ on $X$ is color preserving,  then $G \backslash X$ is also pure $n$-dimensional and partite and $p: X \rightarrow G \backslash X$ is color preserving and in particular,  rigid.
\end{observation}

\subsection{The $A_n$ root system}

\label{A_n subsec}

Here,  we set some definitions and notations regarding the $A_n$ root system that we will use in the sequel.  Our use of the $A_n$  root system terminology is self-contained and we do not assume any familiarity with this subject.  The reader should note that our definitions are somewhat non-standard (although equivalent to the standard definitions that are usually given in the language of vectors).

Let $n \geq 2$.  We consider the \textit{$A_n$ root system} as the set of pairs $\lbrace (i,j) : 1 \leq i, j \leq n+1,  i \neq j \rbrace$,  where each pair $(i,j)$ is called a \textit{root}.  We fix the following definitions:
\begin{enumerate}
\item Two roots $(i,j)$ and $(i',j')$ are called \textit{opposite} if $(i,j) = (j',i')$.
\item The set of roots  $C_0 = \lbrace (i,j) : 1 \leq i < j \leq n+1 \rbrace$ is defined to be the \textit{positive Weyl chamber}  and each pair $(i,j)$,  $1 \leq i < j \leq n+1$ is a \textit{positive root}.
\item The \textit{boundary of $C_0$} is the set $\partial C_0 = \lbrace (i,i+1) : 1 \leq i \leq n \rbrace \subseteq C_0$.
\item For every permutation $\gamma \in \Sym \lbrace 1,..., n+1 \rbrace$,  we define the Weyl chamber $C_{\gamma} = \lbrace (\gamma (i),  \gamma (j)) : 1 \leq i < j \leq n+1 \rbrace$.  With this notation $C_e = C_0$ (where $e \in \Sym \lbrace 1,..., n+1 \rbrace$ is the trivial permutation).  We further define the boundary of $C_{\gamma}$ to be $\partial C_{\gamma} = \lbrace (\gamma (i), \gamma (i+1)) : 1 \leq i \leq n \rbrace$.
\item The set $\lbrace C_{\gamma} : \Sym \lbrace 1,..., n+1 \rbrace \rbrace$ is the set of all Weyl chambers of the system.
\item For every $\gamma,  \gamma ' \in \Sym \lbrace 1,..., n+1 \rbrace$,  we define $\gamma ' . C_{\gamma} = C_{\gamma '\gamma}$ and note that this is in fact an action of $\Sym \lbrace 1,..., n+1 \rbrace$ on the set of Weyl chambers.
\end{enumerate}

\subsection{Coboundary and cosystolic expansion with general group coefficients}

Here we recall the basic definitions regarding $1$-coboundary and $1$-cosystolic expansion of a simplicial complex $X$ with coefficients in a general group $\Lambda$.  These definitions appeared in several works (e.g., \cite{DM}) and we claim no originality here.

Throughout,  we denote $X$ to be a finite simplicial complex that is pure $n$-dimensional (i.e.,  every simplex in $X$ is a face of an $n$-dimensional simplex).   For every $-1 \leq k \leq n$, $X(k)$ denotes the set of $k$-dimensional simplices of $X$ and $X_\ord (k)$ to be ordered $k$-simplices.  We use the convention in which $X(-1) =  \lbrace  \emptyset  \rbrace$.

Given a group $\Lambda$,  we define the space of $k$-cochains for $k=-1,0,1$ with values in $\Lambda$ by
$$C^{-1} (X,  \Lambda) = \lbrace \phi : \lbrace \emptyset \rbrace  \rightarrow \Lambda \rbrace,$$
$$C^{0} (X,  \Lambda) = \lbrace \phi : X(0) \rightarrow \Lambda \rbrace,$$
$$C^1 (X,  \Lambda) = \lbrace \phi :   X_{\ord} (1) \rightarrow \Lambda : \forall (u,v) \in X_{\ord} (1),  \phi ((u,v)) = ( \phi ((v,u)))^{-1} \rbrace.$$
We note that we do not assume that $\Lambda$ is abelian or finite.

For $k=-1,0$,  we define the coboundary maps $d_k : C^{k} (X,  \Lambda) \rightarrow C^{k+1} (X,  \Lambda)$ as follows: First,  for every $\phi \in C^{-1} (X, \Lambda)$ and every $v \in X(0)$,  we define $ d_{-1} \phi (v) = \phi ( \emptyset)$.  Second,  for every $\phi \in C^{0} (X,  \Lambda)$ and every $(v_0, v_1) \in X_\ord (1)$,  we define $d_{0} \phi ((v_0,v_1)) = \phi (v_0) \phi (v_1)^{-1}$.   Last,  we also define the map $d_1 : C^1 (X,  \Lambda) \rightarrow \lbrace \phi : X_\ord (2) \rightarrow \Lambda \rbrace$ as follows: for every $\phi \in C^{1} (X,  \Lambda)$ and every $(v_0, v_1,v_2) \in X_\ord (2)$,  we define
$d_{1} \phi ((v_0,v_1,v_2)) = \phi ((v_0,v_1))  \phi ((v_1,v_2)) \phi ((v_2,v_0))$.

We note that $d_k d_{k-1} \equiv e_\Lambda$ for $k=0,1$ and define cocycles and coboundaries by
$$Z^k (X, \Lambda) = \lbrace \phi \in C^{0} (X,  \Lambda) : d_k \phi \equiv e_\Lambda \rbrace,$$
$$B^k (X, \Lambda)  = d_{k-1} (C^{k-1} (X, \Lambda)).$$
We note that $B^k (X, \Lambda)  \subseteq Z^k (X, \Lambda)$ (as sets).  We further define the an action of $C^{k-1} (X, \Lambda)$ on $Z^{k} (X, \Lambda)$:
\begin{itemize}
\item An element $\psi \in C^{-1} (X, \Lambda)$ acts on $Z^{0} (X, \Lambda)$ by
$$\psi . \phi (v) = \psi (\emptyset)^{-1} \phi (v),  \forall \phi \in Z^{0} (X, \Lambda),  \forall v \in X (0).$$
\item An element $\psi \in C^{0} (X, \Lambda)$ acts on $Z^{1} (X, \Lambda)$ by
$$\psi . \phi ((v_0,v_1)) = \psi (v_0) \phi ((v_0,v_1)) \psi (v_1)^{-1},  \forall \phi \in Z^{1} (X, \Lambda),  \forall (v_0,v_1) \in X_{\ord} (1).$$
\end{itemize}

We denote by $[ \phi ]$ the orbit of $\phi \in Z^{k} (X, \Lambda)$ under the action of $C^{k-1} (X, \Lambda)$ and define the reduced $k$-th cohomology to be
$$H^k (X, \Lambda) = \lbrace [ \phi ] : \phi \in Z^{k} (X, \Lambda) \rbrace.$$
We will say that the $k$-th cohomology is trivial if $H^k (X, \Lambda) = \lbrace [ \phi \equiv e_\Lambda ]  \rbrace$ and note that the cohomology is trivial if and only if $B^k (X, \Lambda) = Z^k (X, \Lambda)$.

When $\Lambda$ is an abelian group,  $Z^k (X, \Lambda),  B^k (X, \Lambda)$ are abelian groups and for $k=0,1$,  the $k$-th reduced cohomology is defined by $H^k (X,  \Lambda) = Z^k (X, \Lambda) / B^k (X, \Lambda)$.  We leave it to the reader to verify that the definition in the abelian setting coincides (as sets) with the definition in the general setting given above.

In order to define expansion, we need to define a norm.  Let $ w : \bigcup_{k=-1}^n X(k) \rightarrow \mathbb{R}_+$ be the weight function defined above.
For $\phi \in C^0 (X, \Lambda)$,  we define $\supp (\phi) \subseteq X(0)$ as
$$\supp (\phi) = \lbrace v : \phi (v) \neq e_\Lambda \rbrace.$$
Also,  for $\phi \in C^{1} (X, \Lambda)$,  we define
$$\supp (\phi) = \lbrace \lbrace u,v \rbrace : \phi ((u,v)) \neq e_\Lambda \rbrace,$$
and
$$\supp (d_1 \phi) =  \lbrace \lbrace u,v, w \rbrace : \phi ((u,v,w)) \neq e_\Lambda \rbrace.$$
We note that for every permutation $\pi$ on $\lbrace 0,1,2\rbrace$ it holds that $d_1 \phi ((v_0,v_1,v_2)) = e_\Lambda$ if and only if $d_1 \phi ((v_{\pi (0)},v_{\pi (1)},v_{\pi (2)})) = e_\Lambda$ and thus $\supp (d_1 \phi)$ if well-defined.

With this notation,  we define the following norms: for $\phi \in  C^{k} (X, \Lambda)$,  $k=0,1$,  we define
$$\Vert \phi \Vert = \sum_{\tau \in \supp (\phi)} w (\tau).$$
Also,  for $\phi \in  C^{1} (X, \Lambda)$, we define
$$\Vert d_1 \phi \Vert = \sum_{\tau \in \supp (d_1 \phi)} w (\tau).$$

Using the notations above,  we will define coboundary and cosystolic expansion constants (generalizing the Cheeger constant) as follows: For $k =0,1$,  define
$$h^{k}_{\cobound} (X, \Lambda) = \min_{\phi \in C^{k} (X, \Lambda) \setminus B^{k} (X, \Lambda)} \dfrac{\Vert d_k \phi \Vert}{\min_{\psi \in B^k (X,\Lambda)} \Vert \phi \psi^{-1} \Vert},$$
$$h^{k}_{\cosys} (X, \Lambda) = \min_{\phi \in C^{k} (X, \Lambda) \setminus Z^{k} (X, \Lambda)} \dfrac{\Vert d_k \phi \Vert}{\min_{\psi \in Z^k (X,\Lambda)} \Vert \phi \psi^{-1} \Vert}.$$
Observe the following (for $k =0,1$):
\begin{itemize}
\item For any $X$,  $h^{k}_{\cosys} (X, \Lambda) > 0$.
\item If $H^{k}  (X, \Lambda)$ is non-trivial, then $h^{k}_{\cobound} (X, \Lambda) =0$.
\item If $H^{k}  (X, \Lambda)$ is trivial,  then $h^{k}_{\cobound} (X, \Lambda) = h^{k}_{\cosys} (X, \Lambda)$.
\end{itemize}

For a constant $\beta >0$,  we say that
\begin{itemize}
\item The complex $X$ is a \textit{$(\Lambda,  \beta)$ 1-coboundary expander},  if $h^{0}_{\cobound} (X, \Lambda) \geq \beta$ and $h^{1}_{\cobound} (X, \Lambda) \geq \beta$.
\item The complex $X$ is a \textit{$(\Lambda,  \beta)$ 1-cosystolic expander},  if $h^{0}_{\cosys} (X, \Lambda) \geq \beta$,  $h^{1}_{\cosys} (X, \Lambda) \geq \beta$ and for every $\phi \in Z^1 (X,  \Lambda) \setminus B^1 (X,  \Lambda)$ it holds that $\Vert \phi \Vert \geq \beta$.
\end{itemize}

We call a family of complexes $\lbrace X^{(s)} \rbrace_s$,  has \textit{$1$-coboundary expansion over $\Lambda$} ( \textit{$1$-cosystolic expansion over $\Lambda$}), if there exists $\beta >0$ such that for every $s$,  the complex $X^{(s)}$ is a $(\Lambda,  \beta)$ 1-coboundary expander ($(\Lambda,  \beta)$ 1-cosystolic expander).

\begin{remark}
A lower bound on $h^{0}_{\cobound}$ is usually easy to verify: If the $1$-skeleton of $X$ is a connected finite graph, then for any non-trivial group $\Lambda$, $h^{0}_{\cobound} (X, \Lambda)$ is exactly the Cheeger constant of the weighted $1$-skeleton (where each edge $\lbrace u,v \rbrace$ is weighted by $w(\lbrace u,v \rbrace$).  Moreover,  the Cheeger constant of the weighted $1$-skeleton can be bounded by the spectral gap of the weighted random walk on the $1$-skeleton, which in concrete examples can be bounded using the trickling down Theorem \cite{OppLocI}.
\end{remark}

In \cite{DD-cosys},  the following result was proven (generalizing previous results in \cite{KKL,  EK} that applied to the Abelian setting, and improving parameters over the result of \cite{KM}, that was the first co-systolic expansion result in the non-Abelian setting):
\begin{theorem}\cite[Theorem 8]{DD-cosys}
\label{DD for cosys thm}
Let $0 \leq \lambda <1,  \beta >0$ be constants and $\Lambda$ be a group.  For a pure $n$-dimensional simplicial complex $X$ with $n \geq 3$,  if for every vertex $v$,  $h^1_{\cobound} (X, \Lambda) \geq \beta$ and $X$ is a $\lambda$-one-sided local spectral expander,  then
$$h^1_{\cosys} (X, \Lambda) \geq \frac{(1-\lambda) \beta}{24} - e \lambda,$$
where $e \approx 2.71$ is the Euler constant.
\end{theorem}

\subsection{Coset complexes and quotients}

For $n \in \mathbb{N},  n \geq 1$,   a group $\Gamma$ with subgroups $K_{i} < \Gamma,  i =0,...,n$,  the \textit{coset complex} $\mathcal{CC} (\Gamma,  \lbrace K_{i}  \rbrace_{0 \leq i \leq n})$ is defined to be the partite $n$-dimensional simplcial complex with vertices $\lbrace g K_i : g \in \Gamma,  0 \leq i \leq n \rbrace$ such that for every $0 \leq k \leq n$ it holds that $\lbrace g_{i_0} K_{i_0},...,   g_{i_k} K_{i_k} \rbrace$ is a $k$-simplex if and only if $0 \leq i_0 < i_1 <...< i_k \leq n$,  $g_{i_j} \in \Gamma$ and for every $0 \leq j, l \leq k$,  $g_{i_j} K_{i_j} \cap g_{i_l} K_{i_l} \neq \emptyset$.

We note that $\Gamma$ acts on $X = \mathcal{CC} (\Gamma,  \lbrace K_{i}  \rbrace_{0 \leq i \leq n})$ by $g' . g K_{i} = g' g K_{i}$ for every $g', g \in \Gamma$ and $0 \leq i \leq n$ and that this action is color preserving and transitive on $X (n)$,  i.e., for every $\sigma, \sigma ' \in X(n)$ there exists $g \in \Gamma$ such that $g. \sigma = \sigma '$.   Similarly,  every subgroup $\Gamma ' < \Gamma$ also acts on $X$ simplicially and color preserving (but for a subgroup $\Gamma '$, the action need not be transitive on $X (n)$.

\begin{proposition}
\label{N backs X is a coset complex prop}
Let $n \in \mathbb{N},  n \geq 1$ and $\Gamma$ with subgroups $K_{i} < \Gamma,  i =0,...,n$.  Denote $X = \mathcal{CC}  (\Gamma,  \lbrace K_{i}  \rbrace_{0 \leq i \leq n })$ to be the coset complex defined above.   For a normal subgroup $N \triangleleft \Gamma$,  $N \backslash X$ is (isomorphic to) the coset complex $ \mathcal{CC}  (N \backslash \Gamma,  \lbrace N \backslash (N K_{i})  \rbrace_{0 \leq i \leq n })$.
\end{proposition}

\begin{proof}
For a vertex $v$ of $X$,  $v = g K_{i}$,  the orbit of $v$ under the action of $N$ is
$$[g K_{i}] = \lbrace h g K_{i} : h \in N \rbrace = \lbrace g (g^{-1} h g) K_{i} : h \in N \rbrace = \lbrace g h K_{i} : h \in N \rbrace = g N K_{i}.$$
We further note that for every $g'$,  if $N g = Ng'$,  then
$$g' N K_{i} = N g' K_{i}  = N g K_{i}  = g N K_{i} = [v].$$
This shows that $[g K_{i}] = (N g) N K_{i}$ and thus the vertices of $N \backslash X$ can be identified with the cosets $g (N \backslash (N K_{i}))$ for $g \in N \backslash \Gamma$.

Let $g_0,...,g_k \in N \backslash \Gamma$ and $0 \leq i_0 < i_1 < ... < i_k \leq n$.  By definition,
$$\lbrace g_0 N K_{i_0},...,g_k N K_{i_k} \rbrace \in (N \backslash X) (k)$$ if and only if there are $h_0,...,h_k \in N$ such that
$$\lbrace g_0 h_0  K_{i_0},...,g_k h_k K_{i_k} \rbrace  \in X (k)$$
and this happens if and only if there are $h_0,...,h_k \in N$ such that for every $0 \leq j < j' \leq k$,  $g_j h_j  K_{i_{j}} \cap g_{j '} h_{j ' } K_{i_{j '}} \neq \emptyset$.  Note that the last condition is equivalent to $g_{j} N K_{i_{j}} \cap g_{j '} N K_{i_{j '}} \neq \emptyset$ for every $0 \leq j < j ' \leq k$.  In other words,  $\lbrace g_0 N K_{i_0},...,g_k N K_{i_k} \rbrace$ spans a $k$-simplex in $N \backslash X$ if and only if $\lbrace g_0 N K_{i_0},...,g_k N K_{i_k} \rbrace$ spans a $k$-simplex in $\mathcal{CC}  (N \backslash \Gamma,  \lbrace N \backslash (N K_{i})  \rbrace_{0 \leq i \leq n})$.  This shows that $N \backslash X$ is isomorphic to the coset complex $ \mathcal{CC}  (N \backslash \Gamma,  \lbrace N \backslash (N K_{i})  \rbrace_{0 \leq i \leq n })$.
\end{proof}

\begin{remark}
Below,  we will usually apply the above Proposition in cases where for every $i$,  $N \cap K_i = \lbrace e \rbrace$.  In these cases,  the quotient map $\Gamma \rightarrow N \backslash \Gamma$ is injective on every $K_i$ and thus we will abuse the notation and denote $K_i$ in lieu of $N \backslash (N K_i)$.
\end{remark}

\subsection{The Kaufman-Oppenheim coset complexes}

\label{The Kaufman-Oppenheim coset complexes subsec}

Let $n \geq 2$ and $R$  a unital commutative ring.  We denote $\SL_{n +1} (R)$ to be the group of $(n+1) \times (n +1)$ matrices with entries in $R$ and determinant $1$.  This group has the subgroup of elementary matrices: For $r \in R$ and indices $1 \leq i, j \leq n +1 ,   i \neq j$,  we define the \textit{elementary matrix} $e_{i,j} (r)$ to be the$(n+1) \times (n +1)$ matrix with $1$'s along the main diagonal,  $r$ in the $(i,j)$-th entry and $0$'s in all other entries.  We denote $\EL_{n +1} (R)$ to be the group of matrices with entries in $R$ that is generated by all elementary matrices.   We note that $\EL_{n+1} (R)$ is always a subgroup of $\SL_{n  +1} (R)$.

Let $p$ be a prime.  We denote $\mathbb{F}_p$ to be the field with $p$ elements and $\mathbb{F}_p [t]$ to be the ring of polynomials over this field.
\begin{theorem} \cite[Lemma 7.2 (1)]{Split}
For every $n \geq 2$,  $\SL_{n +1} (\mathbb{F}_p [t]) = \EL_{n +1} (\mathbb{F}_p [t])$.
\end{theorem}

The work \cite{KO-construction,  KO-construction2} introduced an elementary construction of local spectral expanders using coset complexes.  Here, we will repeat a variation of this construction.  Let $n \geq 2$.  Let $\gamma_0 \in \Sym \lbrace 1,...,n+1 \rbrace$ be the cyclic permutation $\gamma_0 = (1 2 ... n +1 )$.  For every $0 \leq i \leq n$,  define $K_i <  \SL_{n +1} (\mathbb{F}_p [t])$ to be the subgroup $K_i = \langle e_{\gamma_0^i (j),  \gamma_0^i (j+1)} (r) : 1 \leq j \leq n ,  r \in \mathbb{F}_p [t], \deg (r) \leq 3 \rangle$.

More explicitly,  $K_0$ is the group composed of upper-triangular matrices of the form
$$\begin{pmatrix}
1 & r_{1,2} & r_{1,3} & \ldots &  r_{1,n +1} \\
0 & 1 & r_{2,3} & \ldots & r_{2,n +1} \\
\vdots & \vdots  & \ddots & \ldots &\vdots  \\
0 &  \ldots & 0 & 1 & r_{n,n +1} \\
0 &  \ldots & 0 & 0 & 1
\end{pmatrix}$$
where $r_{j,k} \in \mathbb{F}_p [t]$ and $\deg (r_{j,k}) \leq 3 (k-j)$.  We note that for every $0 \leq i \leq n$,  $K_0$ is isomorphic to $K_i$.

For every $s \in \mathbb{N}$,  if $s > 3n$,  then $K_i < \SL_{n +1} (\mathbb{F}_p [t] /  \langle t^s \rangle)$.  For every such $s$,  denote $X_{n,p}^{(s)} = \mathcal{CC} (\SL_{n +1} (\mathbb{F}_p [t] / \langle t^s \rangle),  \lbrace K_i \rbrace_{0 \leq i \leq n}$.  When $n,p$ are obvious from the context, we will denote $X^{(s)}$ in lieu of $X_{n,p}^{(s)}$.  As in \cite{KO-construction,  KO-construction2},  the following holds:
\begin{theorem}
Let $n \in \mathbb{N},  n \geq 2$,  $p$ prime such that $p > n^2$.  Then the family $\lbrace X_{n,p}^{(s)}  : s >3n \rbrace$ is a family of pure $n$-dimensional partite,  uniformly bounded degree simplicial complexes that are $\frac{1}{\sqrt{p}-n}$-one-sided local spectral expanders.  Also,  all the links of all vertices of all the complexes in the family $\lbrace X_{n,p}^{(s)}  : s >3n \rbrace$ are isomorphic to each other.
\end{theorem}

In \cite{KO-cobound},  the following was proven:
\begin{theorem}
For $n \geq 3$ there is $\beta >0$ such that for every odd prime $p$,  every simplicial complex $X^{(s)}=X_{n,p}^{(s)}$ and every vertex $v$,  $h^1_{\cobound} ((X^{(s)})_v,  \mathbb{F}_2) \geq \beta$.
\end{theorem}

Note that in the above Theorem the constant $\beta$ is independent of $p$. Thus by Theorem \ref{DD for cosys thm} and the spectral results of \cite{KO-construction} mentioned above,  it follows that for every sufficiently large $p$,   $\lbrace X_{n,p}^{(s)} : s > 3n \rbrace$ is a family of $\mathbb{F}_2$ 1-cosystolic expanders of uniformly bounded degree.  The Theorem above was proved by a method known as ``the cone method''.  This method was generalized to non-Abelian coefficients in \cite{DD-swap} and using their definition,  one can adapt the proof of \cite{KO-cobound} to the following stronger result:
\begin{theorem}
For $n \geq 3$ there is $\beta >0$ such that for every odd prime $p$,  every group $\Lambda$,  every simplicial complex $X^{(s)}=X_{n,p}^{(s)}$ and every vertex $v$,  $h^1_{\cobound} ((X^{(s)})_v,  \Lambda) \geq \beta$.
\end{theorem}

Combining this Theorem with Theorem \ref{DD for cosys thm} and the spectral results of \cite{KO-construction} mentioned above yields:
\begin{theorem}
\label{cosys exp of KO complexes thm}
For every $n \geq 3$,  there exists $p_0$ such that for every prime $p \geq p_0$ and every group $\Lambda$,  the family $\lbrace X_{n,p}^{(s)} : s > 3n \rbrace$ has uniformly bounded degree and $1$-cosystolic expansion over $\Lambda$.
\end{theorem}

\subsection{Presentation of $\SL_n  (\mathbb{F}_p [t])$}

Let $R$ be a unital ring and $n \geq 2$ an integer.  The group $\St_{n+1} (R)$ is defined as follows via generators and relations: The generating set is
$$\lbrace x_{i,j} (r) : r \in R,  1 \leq i,j \leq n+1,  i \neq j \rbrace$$
with the following relations:
\begin{enumerate}
\item For every $ \leq i,j \leq n+1,  i \neq j$,  $x_{i,j} (0) = e$.
\item For every $r_1, r_2 \in R$ and every $1 \leq i,j \leq n+1,  i \neq j$,  $x_{i,j} (r_1) x_{i,j} (r_2) = x_{i,j} (r_1 + r_2)$.
\item For every $r_1, r_2 \in R$ and every $1 \leq i,j,i',j' \leq n+1,  i \neq j, i' \neq j',  j \neq i',  j' \neq i$,  $[x_{i,j} (r_1),  x_{i',j'} (r_2)] =e$.
 \item For every $r_1, r_2 \in R$ and every $1 \leq i,j,k \leq n+1$ that are pair-wise distinct, $[x_{i,j} (r_1),  x_{j,k} (r_2)] =x_{i,k} (r_1 r_2)$.
\end{enumerate}
We note that the map $\St_{n+1} (R) \rightarrow \EL_{n+1} (R)$ induced by $x_{i,j} (r) \mapsto e_{i,j}$ is a surjective homomorphism.

\begin{observation}
\label{kernel of St (R/I) obs}
For an ideal $I$ of $R$,  it holds that the congruence subgroup $\Ker (\St_{n+1} (R) \rightarrow \St_{n+1} (R / I))$ is the subgroup of $\St_{n+1} (R)$ normally generated by all the elements of the form $x_{i,j} (s)$ where $1 \leq i, j \leq n+1,  i\neq j$ and $s \in I$.
\end{observation}

\begin{theorem}\cite[Lemma 7.2 (2)]{Split}
\label{Stn = SLn in Fp t case thm}
For every $n \geq 3$,  $\SL_{n+1} (\mathbb{F}_p [t]) \cong \St_{n+1} (\mathbb{F}_p [t])$.
\end{theorem}

\begin{theorem}\cite[Main Theorem]{Split}
For $\mu = (\mu_0,  \mu_1) \in \mathbb{Z} / 2 \mathbb{Z} \times (\mathbb{N} \cup \lbrace 0 \rbrace)$ denote $t^\mu = (-1)^{\mu_0} t^{\mu_1}$.  Further denote $\vert \mu \vert = \mu_1$.   For every $n \geq 3$,   the following is a presentation of $\St_{n+1} (\mathbb{Z} [t])$:
Generators
$$\lbrace x_{i,j} (t^\mu) : 1 \leq i,j \leq n+1 , i \neq j,  \mu \in \mathbb{Z} / 2 \mathbb{Z} \times \lbrace 0,...,3 \rbrace \rbrace$$
and relations:
\begin{enumerate}
\item For every $1 \leq i,j \leq n+1$ , $i \neq j$ and every $\mu_1 \in \lbrace 0,...,3 \rbrace$, $x_{i,j} (t^{\mu_1}) x_{i,j} (- t^{\mu_1}) = e$.
\item For every $1 \leq i,j,  i', j' \leq n+1$ , $i \neq j$,  $i ' \neq j '$,  $j \neq i'$,  $i' \neq j$ and every $\mu,  \mu '  \in \mathbb{Z} / 2 \mathbb{Z} \times \lbrace 0,...,3 \rbrace$,
$[x_{i,j} (t^\mu),  x_{i',j'} (t^{\mu'})] =e$.
\item For every $1 \leq i,j,k \leq n+1$ pairwise distinct and every $\mu,  \mu '  \in \mathbb{Z} / 2 \mathbb{Z} \times \lbrace 0,...,3 \rbrace$ with $\vert \mu + \mu ' \vert \leq 3$,  $[x_{i,j} (t^\mu),  x_{j,k} (t^{\mu'})] =x_{i,k} (t^{\mu + \mu '})$.
\end{enumerate}
\end{theorem}

\begin{corollary}
\label{presentation of SL_n+1 F_p [t] coro}
For every $n \geq 3$,   the following is a presentation of $\SL_{n+1} (\mathbb{F}_p [t])$:
Generators
$$\lbrace x_{i,j} (t^\mu) : 1 \leq i,j \leq n+1, i \neq j,  \mu \in \mathbb{Z} / 2 \mathbb{Z} \times \lbrace 0,...,3 \rbrace \rbrace$$
and relations:
\begin{enumerate}
\item For every $1 \leq i,j \leq n+1, i \neq j$,  $x_{i,j} (1)^p =e$.
\item For every $1 \leq i,j \leq n+1, i \neq j$ and every $\mu_1 \in \lbrace 0,...,3 \rbrace$, $x_{i,j} (t^{\mu_1}) x_{i,j} (- t^{\mu_1}) = e$.
\item For every $1 \leq i,j,  i', j' \leq n+1, i \neq j,  i ' \neq j ',  j \neq i',  i' \neq j$ and every $\mu,  \mu '  \in \mathbb{Z} / 2 \mathbb{Z} \times \lbrace 0,...,3 \rbrace$,
$[x_{i,j} (t^\mu),  x_{i',j'} (t^{\mu'})] =e$.
\item For every $1 \leq i,j,k \leq n+1$ pairwise distinct and every $\mu,  \mu '  \in \mathbb{Z} / 2 \mathbb{Z} \times \lbrace 0,...,3 \rbrace$ with $\vert \mu + \mu ' \vert \leq 3$,  $[x_{i,j} (t^\mu),  x_{j,k} (t^{\mu'})] =x_{i,k} (t^{\mu + \mu '})$.
\end{enumerate}
\end{corollary}

In order to prove this corollary,  we will need the following Lemma which will also be useful in the sequel:
\begin{lemma}\cite[Lemma 2]{BissD}
\label{[x,y]^p lemma}
Let $G$ be a group and $x,y \in G$ be elements such that $[x,  [x,y]] =e$.  Then for any integer $p \in \mathbb{N}$,
$[x,y]^p = [x^p, y]$.
\end{lemma}

\begin{proof}
By Theorem \ref{Stn = SLn in Fp t case thm},  it holds that $ \SL_{n+1} (\mathbb{F}_p [t]) \cong \St_{n+1}  (\mathbb{F}_p [t])$ and thus we need to show that the presentation above is a presentation of $\St_{n+1}  (\mathbb{F}_p [t])$.

Consider the exact sequence
$$0 \rightarrow N \rightarrow \St_{n+1} (\mathbb{Z} [t]) \rightarrow \St_{n+1}  (\mathbb{F}_p [t]) \rightarrow 0$$
where the homomorphism $\St_{n+1}  (\mathbb{Z} [t]) \rightarrow \SL_{n+1}  (\mathbb{F}_p [t])$ is induced by the map $x_{i,j} (t^{\mu}) \mapsto e_{i,j} (t^{\mu})$.  The subgroup $N$ is normally generated by elements of the form $(x_{i,j} (\pm t^m))^p$ where $1 \leq i,j \leq n+1, i \neq j$ and $m \in \mathbb{N} \cup \lbrace 0 \rbrace$ and we will show that every such element is in the normal closure of $\lbrace x_{i, j} (1)^p : 1 \leq i,j \leq n+1, i \neq j \rbrace$.  Indeed,  let $N'$ be the normal closure of $\lbrace x_{i, j} (1)^p : 1 \leq i,j \leq n+1, i \neq j \rbrace$ and fix some $1 \leq i,j \leq n+1, i \neq j$ and $m \in \mathbb{N} \cup \lbrace 0 \rbrace$.  Let $1 \leq k \leq n+1$ be such that $k \neq i,  k \neq j$.  Note that $[x_{i,k} (1)^p,  x_{k,j} (\pm t^m)] \in N'$.   Also note that in $\St_{n+1}  (\mathbb{Z} [t])$,  $[x_{i,k} (1),  [x_{i,k} (1),  x_{k,j} (\pm t^{m})]] =e$.  Thus,
\begin{align*}
x_{i,j} (\pm t^{m})^p = [x_{i,k} (1),  x_{k,j} (\pm t^{m})]^p  =^{\text{Lemma } \ref{[x,y]^p lemma}}  [x_{i,k} (1)^p,  x_{k,j} (\pm t^{m})] \in N'
\end{align*}
as needed.
\end{proof}

A variation of the above corollary:

\begin{corollary}
\label{presentation of SL_n+1 F_p [t] coro2}
For every $n \geq 3$,   the following is a presentation of $\SL_{n+1}  (\mathbb{F}_p [t])$:
Generators
$$\lbrace x_{i,j} (r) : 1 \leq i,j \leq n+1, i \neq j,  r \in \mathbb{F}_p [t],  \deg (r) \leq 3 \rbrace$$
and relations:
\begin{enumerate}
\item For every $1 \leq i,j \leq n+1, i \neq j$,  $x_{i,j} (0) = e$ (where $0$ there is the trivial element of $\mathbb{F}_p [t]$).
\item For every $1 \leq i,j \leq n+1, i \neq j$ and every $r_1, r_2 \in \mathbb{F}_p [t]$ of degree $\leq 3$, $x_{i,j} (r_1) x_{i,j} (r_2) = x_{i,j} (r_1 + r_2)$.
\item For every $1 \leq i,j,  i', j' \leq n+1, i \neq j,  i ' \neq j ',  j \neq i',  i' \neq j$ and every $r_1, r_2 \in \mathbb{F}_p [t]$ of degree $\leq 3$,
$[x_{i,j} (r_1),  x_{i',j'} (r_2)] =e$.
\item For every $1 \leq i,j,k \leq n+1$ pairwise distinct and every $r_1, r_2 \in \mathbb{F}_p [t]$ of degree $\leq 3$ with $\deg (r_1 r_2) \leq 3$,  $[x_{i,j} (r_1),  x_{j,k} (r_2)] =x_{i,k} (r_1 r_2)$.
\item For every $1 \leq i,j,k \leq n+1$ pairwise distinct and every $r_1, r_2,  r_1', r_2' \in \mathbb{F}_p [t]$ of degree $\leq 3$,  if $r_1 r_2 = r_1 ' r_2 '$,  then  $[x_{i,j} (r_1),  x_{j,k} (r_2)] =[x_{i,j} (r_1 '),  x_{j,k} (r_2 ')] $.
\end{enumerate}
\end{corollary}

\begin{proof}
Let $\Gamma$ be the abstract group defined above and $\Gamma '$ be the abstract group defined in Corollary \ref{presentation of SL_n+1 F_p [t] coro}.  In order to distinguish between the elements of these groups,  we will denote the generator of $\Gamma$ by $x_{i,j}$ and the generators of $\Gamma '$ by $y_{i,j}$.

Let $\phi ' : \Gamma ' \rightarrow \Gamma$ be the map induced by $y_{i,j} (t^\mu) \mapsto x_{i,j} (t^{\mu})$ and $\phi :\Gamma \rightarrow \SL_{n+1} (\mathbb{F}_p [t])$ be the map induced by $x_{i,j} (r) \rightarrow e_{i,j} (r)$.  We note that both $\phi' $ and $\phi$ are surjective homomorphisms and that by Corollary \ref{presentation of SL_n+1 F_p [t] coro},  $\phi \circ \phi ' : \Gamma ' \rightarrow \SL_{n+1} (\mathbb{F}_p [t])$ is an isomorphism.  It follows that both $\phi$ and $\phi '$ are isomorphisms and in particular that $\Gamma$ is isomorphic to $\SL_{n+1} (\mathbb{F}_p [t])$.
\end{proof}

In the sequel,  it will be useful to think of the presentation of $\SL_{n+1}  (\mathbb{F}_p [t])$ in terms of roots.  We introduce the following definition:
\begin{definition}[Relations indexed by a pair of roots]
\label{relations indexed by roots def}
Fix an odd prime $p$ and $n \geq 2$.  Let $S = \lbrace x_{i,j} (r) : 1 \leq i,j \leq n+1,  i \neq j,  r \in \mathbb{F}_p [t], \deg(r) \leq 3 \rbrace$.  For two non-opposite roots $(i,j), (i',j')$ in the $A_n$ root system,  we will define the $\lbrace (i,j),  (i',j') \rbrace$ relations (on the symbols of $S$) as follows:
\begin{itemize}
\item If $(i,j) = (i',j')$,  the $\lbrace (i,j) \rbrace$ relations are the relations:
\begin{enumerate}
\item $x_{i,j} (0) = e$ (where $0$ there is the trivial element of $\mathbb{F}_p [t]$).
\item For every $r_1, r_2 \in \mathbb{F}_p [t]$ of degree $\leq 3$, $x_{i,j} (r_1) x_{i,j} (r_2) = x_{i,j} (r_1 + r_2)$.
\end{enumerate}
\item If $(i,j) \neq (i',j')$,  $j \neq i'$ and $j' \neq i$,  the $\lbrace (i,j), (i',j') \rbrace$ relations are the relations: For every $r_1, r_2 \in \mathbb{F}_p [t]$ of degree $\leq 3$,  $[x_{i,j} (r_1),  x_{i',j'} (r_2)] =e$.
\item If there are $i'',  k'', j''$ such that $\lbrace (i,j),  (i',j') \rbrace = \lbrace (i'', j''),  (j'', k'') \rbrace$,   the $\lbrace (i'', j''),  (j'', k'') \rbrace$ relations are the relations:
\begin{enumerate}
\item For every $r_1, r_2 \in \mathbb{F}_p [t]$ of degree $\leq 3$ with $\deg (r_1 r_2) \leq 3$,  $[x_{i'',j''} (r_1),  x_{j'',k ''} (r_2)] =x_{i'',k ''} (r_1 r_2)$.
\item For every $r_1, r_2,  r_1', r_2' \in \mathbb{F}_p [t]$ of degree $\leq 3$,  if $r_1 r_2 = r_1 ' r_2 '$,  then  $[x_{i'',j''} (r_1),  x_{j'',k''} (r_2)] =[x_{i'',j''} (r_1 '),  x_{j'',k''} (r_2 ')] $.
\end{enumerate}
\end{itemize}
\end{definition}

With this Definition,  Corollary \ref{presentation of SL_n+1 F_p [t] coro2} can be stated succinctly as:
\begin{corollary}
\label{pres. of SLn roots coro}
For every $n \geq 3$,   the following is a presentation of $\SL_{n+1}  (\mathbb{F}_p [t])$:
The generating set is
$$\lbrace x_{i,j} (r) : 1 \leq i,j \leq n+1, i \neq j,  r \in \mathbb{F}_p [t],  \deg (r) \leq 3 \rbrace$$
and the relations are all the $\lbrace (i,j),  (i',j') \rbrace$ relations for all two non-opposite roots $(i,j), (i',j')$ in the $A_n$ root system.
\end{corollary}

\subsection{Presentation of unipotent groups}

For fixed $n \geq 2$,  $p$ prime and $d \in \mathbb{N} \cup \lbrace 0 \rbrace$,  denote $\Unip_n (\mathbb{F}_[t] ; d)$ to be the subgroup of $\SL_n (\mathbb{F}_p [t])$ generated by the set $\langle e_{i,i+1} (r) : r \in \mathbb{F}_p [t], \deg (r) \leq d \rangle$.
More explicitly,  $\Unip_n (\mathbb{F}_[t] ; 3)$ is exactly the subgroup $K_0$ defined above.

As in \cite{KO-cobound},  this group has the following presentation:
\begin{theorem}
\label{Unip pres thm}
For every $n \geq 4$,  every odd prime $p$ and every $d \in \mathbb{N}$,   $\Unip_n (\mathbb{F}_[t] ; d)$ has the following presentation: Generators
$$\lbrace x_{i,i+1} (r),   : r \in \mathbb{F}_p [t],  \deg (r) \leq d,  1 \leq i \leq i-1 \rbrace$$
and relations:
\begin{enumerate}
\item For every $1 \leq i \leq n-1$,  $x_{i,i+1} (0) =e$.
\item For every $1 \leq i \leq n-1$ and every $r_1,  r_2 \in \mathbb{F}_p [t]$ of degree $\leq d$,  $x_{i,i+1} (r_1) x_{i,i+1} (r_2)=x_{i,i+1} (r_1 + r_2)$.
\item For every $1 \leq i, j \leq n-1$ with $i+1 < j$ and every $r_1,  r_2 \in \mathbb{F}_p [t]$ of degree $\leq d$,
$$[x_{i,i+1} (r_1),  x_{j,j+1} (r_2)] =e .$$
\item For every $1 \leq i  \leq n-2$ with and every $r_1,  r_2,  r_3 \in \mathbb{F}_p [t]$ of degree $\leq d$,
$$[[x_{i, i+1} (r_1),  x_{i+1, i+2} (r_2)],  x_{i, i+1} (r_3) ] =e ,$$
$$[[x_{i, i+1} (r_1),  x_{i+1, i+2} (r_2)],  x_{i+1, i+2} (r_3) ] =e.$$
\item For every $1 \leq i  \leq n-2$ with and every $r_1,  r_2,  r_1 ',  r_2 '\in \mathbb{F}_p [t]$ of degree $\leq d$,  if $r_1 r_2 = r_1 ' r_2 '$,  then
$$[x_{i, i+1} (r_1),  x_{i+1, i+2} (r_2)] = [x_{i, i+1} (r_1 '),  x_{i+1, i+2} (r_1 ' r_2 ').$$
\end{enumerate}
\end{theorem}

\section{Cohomology of quotients of complexes}

\begin{theorem}
\label{N mod X thm}
Let $X$ be a connected pure $2$-dimensional simplicial complex and $N$ be a group acting on $X$.  Assume that $p: X \rightarrow N \backslash X$ is rigid.  For every group $\Lambda$,  if $H^1 (X, \Lambda) = H^1 (N, \Lambda) = 0$,  then $H^1 (N \backslash X,  \Lambda) =0$ (the cohomology of $N$ is taken with the trivial action of $N$ on $\Lambda$).
\end{theorem}

\begin{proof}
Let $\phi \in Z^1 (N \backslash X,  \Lambda)$.  We need to show that $\phi \in B^1 (N \backslash X, \Lambda)$, i.e., that there is $\psi : (N \backslash X )(0) \rightarrow  \Lambda$ such that for every $(u,v) \in (N \backslash X)_{\ord} (1)$,  $\phi ((u,v)) = \psi (u) (\psi (v))^{-1}$.

Note that for every $\lbrace u, v \rbrace \in X(1)$ it holds that $p (u) \neq p (v)$ since $p$ is rigid.  Thus for every  $(u,v) \in X_{\ord} (1)$ it holds that $(p(u), p(v)) \in (N \backslash X)_{\ord} (1)$.   Let $p^* \phi$ be the pullback of $\phi$,  i.e.,  for every $(u,v) \in X_{\ord} (1)$,  $(p^* \phi) ((u,v)) = \phi ((p(u), p(v)))$.  We note that for every $g \in N$ and every $(u,v) \in X_{\ord} (1)$,
$$(p^* \phi) ((g.u, g.v)) =  \phi ((p(g. u), p(g. v)) =  \phi ((p(u), p(v)) = (p^* \phi) ((u, v)).$$

Note that for every $(u,v,w) \in X_\ord (2)$ it holds that $(p (u),  p(v),  p(w)) \in (N \backslash X)_\ord (2)$ (since $p$ is rigid) and thus for every $(u,v,w) \in X_\ord (2)$,
$$d_1 (p^* \phi) ((u,v,w)) = \phi ((p(u), p(v))) \phi ((p(v), p(w)))  \phi ((p(w), p(u))) = e,$$
i.e.,  $p^* \phi \in Z^1 (X, \Lambda)$.  By assumption $H^1 (X, \Lambda) = 0$,  i.e.,  there is $\psi ' \in B^1 (X, \Lambda)$ such that for every $(u, v) \in X_\ord (1)$,  $p^* \phi ((u,v))= \psi ' (u) (\psi ' (v))^{-1}$.

Given any $u, v \in X(0)$,  by the assumption that $X$ is connected,  we can choose an oriented path from $u$ to $v$: $(u, v_1),  (v_1,v_2),...,(v_k, v) \in X_{\ord} (1)$.  For every such a path it holds that
\begin{align*}
& (p^* \phi) ((u,v_1)) (p^* \phi) ((v_1, v_2)) ... (p^* \phi) ((v_k,v)) =   \\
& \psi ' (u) (\psi ' (v_1))^{-1} \psi ' (v_1) (\psi ' (v_2))^{-1} ... \psi ' (v_k) (\psi ' (v))^{-1} = \psi ' (u) (\psi ' (v))^{-1}.
\end{align*}
Thus,  for every $g \in N$,
\begin{align*}
& \psi ' (g. u) (\psi ' (g. v))^{-1} = \\
& (p^* \phi) ((g. u, g. v_1)) (p^* \phi) ((g. v_1, g. v_2)) ... (p^* \phi) ((g. v_k, g. v)) = \\
& (p^* \phi) ((u,v_1)) (p^* \phi) ((v_1, v_2)) ... (p^* \phi) ((v_k,v)) = \psi ' (u) (\psi ' (v))^{-1}.
\end{align*}
In particular,  for every $g_1, g_2 \in N$ and every $v \in X(0)$,
\begin{equation}
\label{prod of psi ' inv under act}
\psi '(g_1.v) (\psi ' ((g_1. (g_2. v)))^{-1} = \psi '(v) (\psi ' (g_2. v))^{-1}.
\end{equation}

Fix $v \in X(0)$ and define $f_{v} : N \rightarrow \Lambda$ by $f_{v} (g) = \psi ' (v) (\psi ' (g. v))^{-1}$ for every $g \in N$.  We note that for every $g_1,  g_2 \in N$ it holds that
\begin{align*}
f_{v} (g_1 g_2) =    \psi ' (v) (\psi ' ((g_1 g_2). v))^{-1} = \\
  \psi ' (v) (\psi ' (g_1.  v))^{-1} \psi '(g_1.v) (\psi ' ((g_1. (g_2. v)))^{-1} =^{\eqref{prod of psi ' inv under act}} \\
  \psi ' (v) (\psi ' (g_1.  v))^{-1} \psi '(v) (\psi ' (g_2. v))^{-1}  = f_v (g_1) f_v (g_2),
\end{align*}
i.e.,  $f_v  : N \rightarrow \Lambda$ is a homomorphism.  By the assumption that $H^1 (N, \Lambda) = 0$ it follows that $f_v \equiv e$,  i.e.,  that for every $g \in N$,  $\psi ' (v)=  \psi ' (g. v)$.  This is true for every $v \in X(0)$,  and thus it follows that $\psi '$ is constant on $N$-orbits of vertices on $X$.  Equivalently,  for every $u \in (N \backslash X) (0)$ and every $v_1, v_2 \in p^{-1} (u)$ it holds that $\psi ' (v_1) = \psi ' (v_2)$,  thus for such $u$ we will denote $\psi ' (p^{-1} (u)) = \psi ' (v)$ for some/any $v \in p^{-1} (u)$.

Define $\psi :  (N \backslash X) (0) \rightarrow \Lambda$ by $\psi (u) = \psi ' (p^{-1} (u))$ for every $u \in (N \backslash X) (0)$.  For every $(u,v) \in (N \backslash X)_{\ord} (1)$,  there is $(u' , v' ) \in X_{\ord} (1)$ such that $p(u') = u,  p (v') =v$.  For such $(u' ,v')$ it holds that
\begin{align*}
\phi ((u,v)) = (p^* \phi) ((u', v')) =  \psi ' (u ') (\psi ' (v'))^{-1} = \psi (u) (\psi  (v))^{-1}  = d_0 \psi ((u,v)),
\end{align*}
i.e.,  $\phi \in B^1 (N \backslash X, \Lambda)$ as needed.
\end{proof}

\begin{corollary}
\label{general coset complex H^1 coro}
Let $n \geq 2$ and $\Gamma$ a group with subgroups $K_i,  0 \leq i \leq n$ and a normal subgroup $N \triangleleft \Gamma$.   For any group $\Lambda$,  if
$H^1 (N, \Lambda) = 0$ and
$$H^1 (\mathcal{CC} (\Gamma,  \lbrace K_i \rbrace_{0 \leq i \leq n} ), \Lambda) =0,$$
then
$$H^1 (\mathcal{CC} (N \backslash \Gamma,  \lbrace N \backslash NK_i \rbrace_{0 \leq i \leq n} ), \Lambda) =0.$$
\end{corollary}

\begin{proof}
Denote $X = \mathcal{CC} (\Gamma,  \lbrace K_i \rbrace_{0 \leq i \leq n} )$.   As noted in Observation \ref{color presev action is rigid obs},  the map $p : X \rightarrow N \backslash X$ is rigid, since the action of $N$ on $X$ is color preserving.  Thus,  by the Theorem above,  $H^1 (N \backslash X, \Lambda) =0$.
By Proposition \ref{N backs X is a coset complex prop},  $N \backslash X$ is isomorphic to $\mathcal{CC} (N \backslash \Gamma,  \lbrace N \backslash NK_i \rbrace_{0 \leq i \leq n})$.
\end{proof}

Below,  we will need the following instantiation of this Corollary:
\begin{corollary}
\label{p coset complex H^1 coro}
Let $n \geq 2$,  $p$ a prime and $\Gamma$ a group with subgroups $K_i,  0 \leq i \leq n$ and a normal subgroup $N \triangleleft \Gamma$.  Also let $\Lambda$ be a group that does not have non-trivial elements of order $p$.  If $N$ is generated by elements of order $p$ and
$$H^1 (\mathcal{CC} (\Gamma,  \lbrace K_i \rbrace_{0 \leq i \leq n} ), \Lambda) =0,$$
then
$$H^1 (\mathcal{CC} (N \backslash \Gamma,  \lbrace N \backslash NK_i \rbrace_{0 \leq i \leq n} ), \Lambda) =0.$$
\end{corollary}

\begin{proof}
By Corollary \ref{general coset complex H^1 coro},  it is enough to show that $H^1 (N , \Lambda) =0$.  By definition,  this is equivalent to showing that the only homomorphism $\phi : N \rightarrow \Lambda$ is the trivial homomorphism.  Let $\phi : N \rightarrow \Lambda$ and let $h \in N$ of order $p$ (i.e.,   $h^p = e$).  Then $\phi (h)^p =e$ and since $\Lambda$ does not have non-trivial elements, it follows that $\phi (h) =e$.  By our assumption,  $N$ is generated by elements of order $p$ and thus $\phi$ maps a generating set of $N$ to the identity and thus $\phi$ is the trivial map.
\end{proof}

\section{Chamber and pre-chamber groups}

For $n \geq 3$, $p$ prime,  we define the abstract group $G_{n,p}$,  that we will refer to as \textit{$(n,p)$-chamber group} (or \textit{chamber group} in case $n,p$ are clear from the context)  as the "uni-upper-triangular" subgroup that arise from the presentation of $\SL_n (\mathbb{F}_p [t])$ given in Corollary \ref{pres. of SLn roots coro}.

Recall that for $n \geq 2$,  $C_0$ was defined as the Weyl chamber $\lbrace (i,j) : 1 \leq i < j \leq n+1 \rbrace$ of the $A_n$ root system.  Define the $(n,p)$-chamber group to be the group $G_{n,p}$ with generators
$$\lbrace x_{i,j} (r) : 1 \leq i < j \leq n+1,  r \in \mathbb{F}_p [t], \deg (r) \leq 3 \rbrace$$
and the relations are all the $\lbrace (i,j), (i',j') \rbrace$ relations of Definition \ref{relations indexed by roots def} for every two roots $(i,j), (i',j') \in C_0$.

A group $G$ will be called a  \textit{ pre-$(n,p)$-chamber group} (or a \textit{pre-chamber group} when $n,p$ are clear from the context) if it has a generating set
$$\lbrace x_{i,j} (r) : 1 \leq i < j \leq n+1,  r \in \mathbb{F}_p [t], \deg (r) \leq 3 \rbrace$$
and the following conditions are satisfied:
\begin{itemize}
\item For every $(i,j), (i',j), (i,j') \in C_0$,  the $\lbrace (i,j), (i,j') \rbrace$ relations and the $\lbrace (i,j), (i',j) \rbrace$ relations are fulfilled in $G$.
\item For every $(i,j), (i',j') \in \partial C_0$,  the $\lbrace (i,j), (i',j') \rbrace$ relations are fulfilled in $G$.
\item The map $G \rightarrow G_{n,p}$ induced by the natural map of the generating set is a surjective homomorphism.
\end{itemize}

\begin{remark}
We note that the relations stated above for $G$ need not be the defining relations of $G$ and,  in particular,  every chamber group is also a pre-chamber group.
\end{remark}

\begin{remark}
The last condition in the definition of a pre-chamber group is meant to ensure that $G$ does not have relations that do not occur in a chamber group.
\end{remark}

\begin{theorem}
\label{N_i thm}
Let $p$ be an odd prime,  $n \geq 2$ and $G$ a pre-$(n,p)$-chamber group.  There are groups $G=G^0, G^1,...,G^{k} = G_{n,p}$ (where $G_{n,p}$ is an $(n,p)$-chamber group) and $N_0 \triangleleft G^0,...,N_{k-1} \triangleleft G^{k-1}$ such that for every $0 \leq i \leq k-1$, $G^{i+1} =  N_{i} \backslash G^{i}$ and $N_i$ is generated by elements of order $p$.
\end{theorem}

\begin{proof}

The proof is by induction on $n$.  We note that for $n=2$,  the relations in the definition of a  pre-$(2,p)$-chamber group. group are exactly the relations in the definition of a $(2,p)$-chamber group and thus every  pre-$(2,p)$-chamber group. group is a $(2,p)$-chamber group.

 Let $n \geq 3$ and assume that the Theorem holds for $n-1$.   Let $G$ be a  pre-$(n,p)$-chamber group.  Let $G' < G$ be the subgroup of $G$ generated by elements of the form $\lbrace x_{i,j} (r) : 1 \leq i <j \leq n,  r \in \mathbb{F}_p [t], \deg (r) \leq 3 \rbrace$.  We note that $G'$ is a pre-$(n,p)$-chamber group and thus by the induction assumption,  there are groups $G ' =(G^0) ', (G^1) ',...,(G^{k '}) ' = G_{n-1,p}$ and $N_0 ' \triangleleft (G^0) ',...,N_{k'-1} ' \triangleleft (G^{k'-1}) '$ such that for every $0 \leq i \leq k'-1$, $(G^{i+1}) ' =   N_{i} ' \backslash (G^{i}) '$ and $N_i '$ is generated by elements of order $p$.  Let $N_0,...,N_{k' -1}$ be the normal closures of $N_0 ',...,N_{k'-1} '$ in $G$.  Note that $N_0,...,N_{k'-1}$ are also generated by elements of order $p$ and define $G^0 = G,  G^{1} = N_0 \backslash G^0 ,...,G^{k'} = N_{k'-1} \backslash G^{k'-1}  $.  Note that $G^{k'}$ is  pre-$(n,p)$-chamber group and that in $G^{k'}$ all the $\lbrace (i,j),  (i',j') \rbrace$-relations holds for every $(i,j), (i',j') \in C_0$ with $1 \leq i < j \leq n$ and $1 \leq i ' < j ' \leq n$.

Similarly, let $G'' < G^{k'}$ be the subgroup generated by elements of the form $\lbrace x_{i,j} (r) : 2 \leq i <j \leq n+1,  r \in \mathbb{F}_p [t], \deg (r) \leq 3 \rbrace$.  We note again that $G''$ is a pre-$(n-1,p)$-chamber group and thus there are groups $G '' =(G^0) '', (G^1) '',...,(G^{k ''}) ' = G_{n-1,p}$ and $N_0 '' \triangleleft (G^0) '',...,N_{k''-1} '' \triangleleft (G^{k''-1}) '$ such that for every $0 \leq i \leq k''-1$, $(G^{i+1}) '' =N_{i} '' \backslash (G^{i}) '' $ and $N_i ''$ is generated by elements of order $p$.  Denote $N^{k'},...,N^{k' + k''-1}$ to be the normal closures of $N_0 '',...,N_{k''-1} ''$ in $G^{k'}$ (again, these are normal subgroups generated by elements of order $p$) and define $G^{k'+i} = N_{k'+i-1} \backslash G^{k' +i-1}  $ for every $1 \leq i \leq k''$.  Then $G^{k' + k''}$ is a  pre-$(n,p)$-chamber group in which all the $\lbrace (i,j),  (i',j') \rbrace$-relations holds for every $(i,j), (i',j') \in C_0$  are either $1 \leq i < j \leq n$ and $1 \leq i' < j' \leq n$ (from the previous step) or $2 \leq i < j \leq n+1$ and $2 \leq i' < j' \leq n+1$ (from this step).

Denote $H < G^{k' + k'}$ to be the subgroup generated by $\lbrace x_{i,i+1} (r) : 1 \leq i \leq n,  r \in \mathbb{F}_{p} [t],  \deg (r) \leq 3 \rbrace$.  Note that by Theorem \ref{Unip pres thm},  $H$ is isomorphic to $\Unip_{n+1} (\mathbb{F}_p [t] ; 3)$.  We will use $H$ to verify that  $G^{k' + k''}$ has almost all the relations of an $(n,p)$-chamber group: it is missing only $\lbrace (i,j),  (i',j') \rbrace$-relations where either $(1,n+1) \in \lbrace (i,j),  (i',j') \rbrace$ or $\lbrace (i,j),  (i',j') \rbrace = \lbrace (1,j), (j,n+1) \rbrace$ (and in this case, we know some of the relations, but not all of them).

As noted above,  in $ G^{k' + k''}$,  we already know all the $\lbrace (i,j),  (i',j') \rbrace$-relations holds for every $(i,j), (i',j') \in C_0$  where either $1 \leq i < j \leq n$ and $1 \leq i' < j' \leq n$  or $2 \leq i < j \leq n+1$ and $2 \leq i' < j' \leq n+1$.  Below, we will show that $\lbrace (i,j),  (i',j') \rbrace$-relations in the case $\lbrace (i,j),  (i',j') \rbrace = \lbrace (1,j),  (i, n+1) \rbrace$ where $1 \leq j \leq n$ and $2 \leq i \leq n+1$.

Let $(1,j),  (i,n+1) \in C_0$  and $r_1, r_2 \in \mathbb{F}_p [t]$ of degree $\leq 3$.   Recall the notation $[a_1,...,a_l] = [...[a_1,a_2],a_3],...,a_l]$.  Note that
$$x_{1,j} (r_1) = [x_{1,2} (r_1), x_{2,3} (1),..., x_{j-1,j} (1) ],$$
$$x_{i,n+1} (r_2) =  [x_{i,i+1} (r_2), x_{i+1,i+2} (1),..., x_{n, n+1} (1) ].$$
It follows that $x_{1,j} (r_1),  x_{i,n+1} (r_2)$ are elements in $H$ which is isomorphic to $\Unip_{n+1} (\mathbb{F}_p [t] ; 3)$.  Thus every relation that holds in $\Unip_{n+1} (\mathbb{F}_p [t] ; 3)$ between $e_{1,j} (r_1)$ and $e_{i,n} (r_2)$ holds in $H$ between $x_{1,j} (r_1)$ and $x_{i,n} (r_2)$.  This argument shows that if $j \neq i$,  then $\lbrace (1,j),  (i,n+1) \rbrace$-relations.  It also shows that
$$[x_{1,j} (r_1),  x_{j,n} (r_2)], x_{1,j} (r_1)]=e$$
(since this relation holds in $\Unip_{n+1} (\mathbb{F}_p [t] ; 3)$).

We will also show that in  $G^{k' + k''}$,  $\lbrace (1,j),  (j,n+1) \rbrace$-relations partially hold.  Namely,  the relations numbered $2.$ in Definition \ref{relations indexed by roots def} holds.  Let $r_1, r_2, r_1 ',  r_2 ' \in \mathbb{F}_p [t]$ of degree $\leq 3$,  such that $r_1 r_2 = r_1 ' r_2 '$.  We note that in $\Unip_{n+1} (\mathbb{F}_p [t] ; 3)$ it holds that
$$[e_{1,j} (r_1),  e_{j,n+1} (r_2)] =  [e_{1,j} (r_1 '),  e_{j,n+1} (r_2 ')].$$
Thus, by the same argument as above (the fact that $x_{1,j} (r_1),...,x_{j,n+1} (r_2 ')$ are all in $H$), it also holds that
$$[x_{1,j} (r_1),  x_{j,n+1} (r_2)] =  [x_{1,j} (r_1 '),  x_{j,n+1} (r_2 ')].$$

Last,  we define $N_{k' + k''}$ to be the normal subgroup of $G^{k' + k''}$ generated by all the elements of the form
$$[x_{1,j} (r_1),  x_{j,n+1} (r_2)] x_{1,n+1} (- r_1 r_2),$$
where $2 \leq j \leq n$ and $r_1, r_2 \in \mathbb{F}_p [t]$ with $\deg (r_1 r_2) \leq 3$.  We will show that in $G^{k' + k''}$ it holds that
$$([x_{1,j} (r_1),  x_{j,n+1} (r_2)] x_{1,n+1} (- r_1 r_2))^p = e$$
and thus $N_{k'+k''}$ is generated by elements of order $p$.
We note that by the definition of a pre-chamber group,  the $\lbrace (1,j),  (1,n+1) \rbrace$-relations and $\lbrace (j,n+1),  (1,n+1) \rbrace$-relations hold and thus $[x_{1,j} (r_1), x_{1,n+1} (- r_1 r_2)]=e$ and $[x_{j,n+1} (r_2), x_{1,n+1} (- r_1 r_2)] =e$.  Thus
\begin{align*}
([x_{1,j} (r_1),  x_{j,n+1} (r_2)] x_{1,n+1} (- r_1 r_2))^p =
[x_{1,j} (r_1),  x_{j,n+1} (r_2)]^p x_{1,n+1} (- r_1 r_2)^p =
[x_{1,j} (r_1),  x_{j,n+1} (r_2)]^p
\end{align*}
and we are left to show $[x_{1,j} (r_1),  x_{j,n+1} (r_2)]^p =e$.   We showed above that  $[x_{1,j} (r_1),  x_{j,n+1} (r_2)], x_{1,j} (r_1)]=e$.  It follows from Lemma \ref{[x,y]^p lemma} that
\begin{align*}
[x_{1,j} (r_1),  x_{j,n+1} (r_2)]^p = [x_{1,j} (r_1)^p,  x_{j,n+1} (r_2)] = [e,  x_{j,n+1} (r_2)] =e
\end{align*}
as needed.

Let $G^{k'+k''+1} =N_{k' + k''} \backslash G^{k' + k''}  $.  By the definition of the generators of $N_{k' + k''}$, it holds for every $2 \leq j \leq n$ and every $r_1, r_2 \in \mathbb{F}_p [t]$ with $\deg (r_1 r_2) \leq 3$ that
$$[x_{1,j} (r_1),  x_{j,n+1} (r_2)] =  x_{1,n+!} ( r_1 r_2)$$
and thus it follows that $\lbrace (1,j),  (j,n+1) \rbrace$-relations hold (recall that we showed above that these relations partly hold in $G^{k' + k''} $).

To finish the proof,  we need to verify the $\lbrace (i,j), (1,n+1) \rbrace$-relations,  i.e.,  to show that for every $r_1, r_2 \in \mathbb{F}_p [t]$ of degree $\leq 3$ it holds that $[x_{i,j} (r_1),  x_{1,n+1} (r_2) ]=e$.   We note that by the relations we already proved,  it follows that the set $\lbrace \lbrace x_{i,i+1} (r) : 1 \leq i \leq n,  r \in \mathbb{F}_p [t],  \deg (r) \leq 3 \rbrace$ generates $G^{k'+k''+1} $.  Thus it is enough to show that for every $1 \leq i \leq n$,  $[x_{i,i+1} (r_1),  x_{1,n+1} (r_2)] =e$.  Fix $1 \leq i \leq n$ and let $1 \leq j \leq n$ such that $j \neq i$ and $j \neq i+1$ (recall that $n \geq 3$ and thus there exists such $j$).  We already showed that: $[x_{i,i+1} (r_1),  x_{1,j} (1)] =e$ and $[x_{i,i+1} (r_1),  x_{j,n+1} (r_2)] =e$ in $G^{k'+k''+1}$ .   Thus
$$[x_{i,i+1} (r_1),  x_{1,n+1} (r_2) ] = [x_{i,i+1} (r_1),  [x_{1,j} (1), x_{j,n+1} (r_2)] ] = e,$$
as needed.
\end{proof}

Our motivation for proving the Theorem above is the following consequence:
\begin{theorem}
\label{making pre chamber into chamber and vanishing of H1 thm}
Let $n \geq 3$,  $p$ an odd prime and $\Gamma$ a group with subgroups $K_i, 0 \leq i \leq n$ and a subgroup $G < \Gamma$ that is a pre-$(n,p)$-chamber group.  Also let $\Lambda$ be a group that does not have elements of order $p$.  Denote $N$ to be the normal closure of $\Ker (G \rightarrow G_{n,p})$ in $\Gamma$ (where $G_{n,p}$ denotes a $(n,p)$-chamber group).  Assume that for $X = \mathcal{CC} (\Gamma,  \lbrace K_i \rbrace_{0 \leq i \leq n})$,  it holds that $H^1 (X,  \Lambda) =0$.  Then for $X ' = \mathcal{CC} (N \backslash \Gamma,  \lbrace N \backslash NK_i \rbrace_{0 \leq i \leq n})$ it holds that $H^1 (X',  \Lambda) =0$.
\end{theorem}

\begin{proof}
Let $N_0,...,N_{k-1}$ be the subgroups of Theorem \ref{N_i thm}.  Denote $\Gamma^0 = \Gamma$ and $\Gamma^i = \langle\langle N_{i-1} \rangle\rangle \backslash \Gamma^i$, where $ \langle\langle N_{i-1} \rangle\rangle$ denotes the normal closure of $N_{i-1}$ in $\Gamma^i$.  Also denote $X_i = \mathcal{CC} (\Gamma_i,  \lbrace \langle\langle N_{i-1} \rangle\rangle \backslash (\langle\langle N_{i-1} \rangle\rangle K_i) \rbrace_{0 \leq i \leq n}$.  Then $\Gamma^k = N \backslash \Gamma$ and $X_k = X'$.

We note that since each $N_i$ is generated by elements of order $p$, it follows that $ \langle\langle N_{i} \rangle\rangle$ is generated by elements of order $p$.  Applying Corollary \ref{p coset complex H^1 coro} iteratively yields that for every $i$,  $H^1 (X_i,  \Lambda) = 0$ and in particular $H^1 (X',  \Lambda) =0$.
\end{proof}

\section{Propagation of properties in the $A_n$ root system}
\label{Propagation of property in the A_n root system sec}
In this section, we describe a combinatorial result regarding a propagation argument in the $A_n$ root system.

Below,  given a pair of roots $(i,j),  (i',j')$ and a set of Weyl chambers $\mathcal{C}$,  we will denote $\lbrace (i,j),  (i',j') \rbrace < \mathcal{C}$,  if there is a chamber $C \in \mathcal{C}$ such that $\lbrace (i,j),  (i',j') \rbrace \subseteq C$.

\begin{theorem}
\label{propagation thm}
Let $n \geq 3$ and $\gamma_0 \in  \Sym \lbrace 1,..., n+1 \rbrace$ be the cyclic permutation $\gamma_0 = (1 ...  n+1)$.  We define sets $\mathcal{C}_k$ of Weyl chambers inductively as follows:
\begin{itemize}
\item $$\mathcal{C}_0 = \lbrace C_{\gamma_0^i} : 0 \leq i \leq n \rbrace.$$
\item For $k>0$,  $C_\gamma \in \mathcal{C}_k$ if the following conditions holds:
\begin{enumerate}
\item For any $(i,j),  (i',j') \in \partial C_\gamma$,  $\lbrace (i,j),  (i',j') \rbrace < \mathcal{C}_{k-1}$.
\item For every $(i,j),  (i,j'),  (i',j) \in C_\gamma$,  $\lbrace (i,j),  (i,j') \rbrace < \mathcal{C}_{k-1}$ and $\lbrace (i,j),  (i',j) \rbrace < \mathcal{C}_{k-1}$.
\end{enumerate}
\end{itemize}
Then for every two non-opposite roots $(i,j),  (i',j')$,  $\lbrace (i,j),  (i',j') \rbrace < \mathcal{C}_2$.
\end{theorem}

\begin{remark}
The reader should think about this Theorem as the following propagation result: Let $(P)$ be a property of Weyl chambers that is checked on pairs of roots,  i.e., a chamber $C_{\gamma}$ has the property $(P)$ if every pair of roots $(i,j),  (i' ,j ') \in  C_{\gamma}$ has the property $(P)$.  Start with the set $\mathcal{C}_0$ of Weyl chambers defined above that has the property $(P)$ and assume that $(P)$ is inherited from a subset of pairs of roots,  i.e.,  if a ``generating'' subset of pairs of roots in $C_{\gamma}$ has the property $(P)$ at the $k-1$ step,  then  $C_{\gamma}$ has  the property $(P)$ at step $k$ (where the ``generating'' subset of pairs is the set defined in the Theorem).  The Theorem above states that (with $\mathcal{C}_0$ as above) after $2$ steps,  all pairs of non-opposite roots have the property $(P)$.
\end{remark}

In order to prove this Theorem,  we will need few Lemmas:
\begin{lemma}
\label{(i,j), (i,j') lemma}
For any roots of the form $(i,j), (i,j'),  (i',j)$,  $\lbrace (i,j),  (i,j') \rbrace < \mathcal{C}_0$ and $\lbrace (i,j),  (i',j) \rbrace < \mathcal{C}_0$.
\end{lemma}

\begin{proof}
Let $(i,j), (i,j'),  (i',j)$ be roots.  Note that
$$\gamma_0^{1-i}. (i,j) = (1,  \gamma_0^{1-i}. j) \in C_0,$$
$$\gamma_0^{1-i}. (i,j') = (1,  \gamma_0^{1-i}. j') \in C_0.$$
It follows that $(i,j),  (i,j') \in C_{\gamma_0^{i-1}}$ as needed.

Similarly,
$$\gamma_0^{n+1-j}. (i,j) = (\gamma_0^{n+1-j}. i,  n+1) \in C_0,$$
$$\gamma_0^{n+1-j}. (i',j) = (\gamma_0^{n+1-j}. i' ,  n+1) \in C_0.$$
It follows that $(i,j),  (i,j') \in C_{\gamma_0^{j}}$ as needed.
\end{proof}

\begin{lemma}
\label{(i,i+1), (i',j') lemma}
For every $1 \leq i \leq n$ and every root $(i',j') \neq (i+1,i)$,  $ \lbrace (i',j'), (i,i+1) \rbrace < \mathcal{C}_0$.
\end{lemma}

\begin{proof}
Let $(i',j')$ be a root and $1 \leq i \leq n$.  If $i' =1$ or $j' =n+1$,  then $(i',j'),  (i+1) \in C_0$ and we are done.  Assume that $2 \leq i' \leq n+1$ and $1 \leq j ' \leq n$.

We note that for every $0 \leq l \leq n$,  if $l \neq n+1-i$ it holds that $\gamma_0^l.  (i,i+1) \in C_0$.  We also note that for every root $(i',j')$ it holds that
$$\gamma_0^{n+2-i'}.(i', j') = (1, \gamma_0^{n+2-i'}.j') \in C_0$$
and
$$\gamma_0^{n+1-j'}.(i', j') = (\gamma_0^{n+1-j'}. i', n+1) \in C_0.$$
By the assumption that $(i',j') \neq (i+1,i)$,  then  either $n+2-i' \neq n+1-i$ or $n+1-j' \neq  n+1-i$ (or both).  It follows that  there is $0 \leq l \leq n$ such that $\gamma_0^l. (i,i+1),  \gamma_0^l. (i',j') \in C_0$ and thus $(i,i+1),  (i',j') \in C_{\gamma_0^{-l}}$ as needed.
\end{proof}

\begin{lemma}
\label{boundary of gamma1 lemma}
Let $\gamma_1$ be the cyclic permutation $\gamma_1 = \begin{pmatrix}
n & n-1 & ... &1
\end{pmatrix} $.  For every $1 \leq j \leq n-1$,
$$\partial C_{\gamma_1^l} = \lbrace (i,  i+1) : 1 \leq i \leq n-1,  i \neq n - l \rbrace \cup \lbrace (n,1) \rbrace \cup \lbrace (n-l,  n+1) \rbrace.$$
\end{lemma}

\begin{proof}
Fix $1 \leq l \leq n-1$.  We will calculate the boundary by considering all the cases of $\gamma_1^l. (i,i+1)$ where $1 \leq i \leq n$:
\begin{enumerate}
\item For $i=l$,  $\gamma_1^l.  (l,l+1) = (n,1)$.
\item For $i = n$,  $\gamma_1^l.  (n,n+1) = (n-l,n+1)$.
\item For $l < i <n$,  $\gamma_1^l.  (i,i+1) = (i-l, i+1-l)$ and thus
$$\gamma_1^l.  \lbrace (i,i+1) : l < i <n \rbrace = \lbrace (i,i+1) : 1 \leq i <n-l \rbrace.$$
\item For $1 \leq i< j$,  $\gamma_1^l.  (i,i+1) = (n-i-j,  n+1-i-j)$ and thus
$$\gamma_1^l.  \lbrace (i,i+1) : 1 \leq i< l \rbrace = \lbrace (i,i+1) : n-l < i \leq n-1 \rbrace.$$
\end{enumerate}
\end{proof}

\begin{lemma}
\label{(i,j),  (i',j') lemma}
Let $\gamma_0 = (1 ... n+1)$ and $\gamma_1 = \begin{pmatrix} n & n-1 & ... &1 \end{pmatrix} $.  For every two non-opposite roots $(i,j),  (i',j')$,  there are $0 \leq t \leq n$ and $0 \leq l \leq n-1$ such that $(i,j),  (i',j') \in C_{\gamma_0^{t} \gamma_1^l}$.
\end{lemma}

\begin{proof}
Assume first that $i' \neq j$.  Denote $i'' = \gamma_0^{-j}.i'$ and note that $i'' \neq n+1$.  Then
$$\gamma_1^{i''-1} \gamma_0^{-j}.(i,j)= (\gamma_1^{i''} \gamma_0^{-j}.i, n+1) \in C_0$$
and
$$\gamma_1^{i''-1} \gamma_0^{-j}.(i',j')= (1,  \gamma_1^{i''-1} \gamma_0^{-j}. j') \in C_0.$$
Thus $(i,j),  (i',j') \in C_{\gamma_0^{-j} \gamma_1^{i''-1}}$.

If $i' = j$, then from the assumption that the roots are non-opposite it follows that $j \neq i'$.  Denote $i'' = \gamma_0^{-j'}.i$ and note that $i'' \neq n+1$.  Then
$$\gamma_1^{i''-1} \gamma_0^{-j'}.(i',j')= (\gamma_1^{i''} \gamma_0^{-j'}.i ' , n+1) \in C_0$$
and
$$\gamma_1^{i''-1} \gamma_0^{-j'}.(i,j)= (1,  \gamma_1^{i''-1} \gamma_0^{-j'}. j) \in C_0.$$
Thus $(i,j),  (i',j') \in C_{\gamma_0^{-j} \gamma_1^{i''-1}}$.
\end{proof}

\begin{proof}[Proof of Theorem \ref{propagation thm}]
In light of Lemma \ref{(i,j), (i,j') lemma},  in order to show that $C_\gamma \in \mathcal{C}_k$,  it is enough to show that for every two roots $(i,j),  (i',j') \in \partial C_{\gamma}$,  there is $C_{\gamma '} \in \mathcal{C}_{k-1}$ such that $\lbrace (i,j),  (i',j')  \rbrace \subseteq C_{\gamma '}$.

By induction on $k$, we will show that every $k$ and every $0 \leq t \leq n$, if $C_{\gamma} \in \mathcal{C}_k$,  then $C_{\gamma_0^t \gamma} \in \mathcal{C}_k$ or in other words,  $\mathcal{C}_k$ is invariant under the action of $\gamma_0$.   For $k=0$,  this is clear.  Assume that this holds for $\mathcal{C}_{k-1}$ and let $C_{\gamma} \in \mathcal{C}_k$.  Let $0 \leq t \leq n$.  Note that roots in $\partial C_{\gamma_0^t \gamma}$ are of the form $\gamma_0^t. (i,j)$,  where $(i,j) \in \partial C_{\gamma}$.
For any $(i,j),  (i',j') \in \partial C_{\gamma}$,  there is $C_{\gamma '} \in \mathcal{C}_{k-1}$ such that $\lbrace (i,j),  (i',j')  \rbrace \subseteq C_{\gamma '}$. Thus
\begin{align*}
\lbrace \gamma_0^t.(i,j),  \gamma_0^t.  (i',j')  \rbrace \subseteq \gamma_0^t.  C_{\gamma'} =  C_{\gamma_0^t \gamma'} \in \mathcal{C}_{k-1}.
\end{align*}
It follows that $C_{\gamma_0^t \gamma} \in \mathcal{C}_k$ as needed.

We denote $\gamma_1 = \begin{pmatrix}
n & n-1 & ... &1 \end{pmatrix} $ as above.  We claim that for every $1 \leq l \leq n-2$,  $C_{\gamma_1^l} \in \mathcal{C}_1$ (recall that $n \geq 3$ and thus there is some $1 \leq j \leq n-2$).  By Lemma \ref{boundary of gamma1 lemma},
$$\partial C_{\gamma_1^l} = \lbrace (i,  i+1) : 1 \leq i \leq n-1,  i \neq n - l \rbrace \cup \lbrace (n,1) \rbrace \cup \lbrace (n-l,  n+1) \rbrace.$$

By Lemma \ref{(i,i+1), (i',j') lemma},  for every pair root of the form $\lbrace (i,i+1), (i',j') \rbrace \subseteq \partial C_{\gamma_1^l}$,  it holds that $ \lbrace (i',j'), (i,i+1) \rbrace < \mathcal{C}_0$..  Thus, in order to show that $C_{\gamma_1^l} \in \mathcal{C}_1$,  we are left to show that $(n,1),  (n-l,n+1) < \mathcal{C}_0$.

We note that $1 \leq l \leq n-2$ and thus $n-l \geq 2$.  It follows that $\gamma_0^{-1}. (n-l,  n+1) = (n-l-1,n) \in C_0$ and $\gamma_0^{-1}. (n,1) = (n-2,n+1) \in C_0$.  Thus $(n-l,  n+1),  (n,1) \in C_{\gamma_0}$.  It follows that for every $0 \leq i \leq n$ and every $0 \leq l \leq n-2$,  $C_{\gamma_0^i \gamma_1^l} \in \mathcal{C}_1$.

Next,  we will show that $C_{\gamma_1^{n-1}} \in \mathcal{C}_2$.  As above,  by Lemmas \ref{(i,i+1), (i',j') lemma},  \ref{boundary of gamma1 lemma},    it is enough to show that there is $C_{\gamma '} \in \mathcal{C}_1$ such that $(n,1), (1,n+1) \in C_{\gamma '}$.

Note that $\gamma_1.(1,2) = (n,1)$ and that $\gamma_1. (2,n+1) = (1,n+1)$ and thus $(n,1),  (1,n+1) \in C_{\gamma_1} \in \mathcal{C}_1$ as needed.

Last,  we showed above that for every $0 \leq l \leq n-1$ and every $0 \leq t \leq n$,  $C_{\gamma_1^l \gamma_0^t} \in \mathcal{C}_2$.  By Lemma \ref{(i,j),  (i',j') lemma},  it follows that for every two non-opposite roots $(i,j),  (i',j')$,  it holds that $\lbrace (i,j), (i',j') \rbrace < \mathcal{C}_2$.
\end{proof}

\section{Vanishing of the first cohomology for Kaufman-Oppenheim complexes}

In this section, we will prove the main results of this paper.

Let $n \geq 3$ and $p$ an odd prime.  Keeping the notations of the previous Section,  we denote $\gamma_0$ to be the cyclic permutation $\gamma_0 = (1 ... n+1)$,  and $\mathcal{C}_0 = \lbrace C_{\gamma_0^l} : 0 \leq l \leq n \rbrace$.  Also,  for a pair of roots $(i,j),  (i',j')$ and a set of Weyl chambers $\mathcal{C}$,  we will denote $\lbrace (i,j),  (i',j') \rbrace < \mathcal{C}$,  if there is a chamber $C \in \mathcal{C}$ such that $\lbrace (i,j),  (i',j') \rbrace \subseteq C$.

We define the group $\widetilde{\Gamma}_{n,p}$ to be the following abstract group: The generating set of $\widetilde{\Gamma}_{n,p}$ is
$$\lbrace x_{i,j} (r) : 1 \leq i,j \leq n+1, i \neq j,  r \in \mathbb{F}_p [t],  \deg (r) \leq 3 \rbrace$$
and its relations are the $\lbrace (i,j),  (i',j') \rbrace$-relations of Definition \ref{relations indexed by roots def} for every $\lbrace (i,j),  (i',j') \rbrace < \mathcal{C}_0$.

For every $0 \leq i \leq n$,  we define the subgroup $\widetilde{K}_i < \widetilde{\Gamma}_{n,p}$ as
$$\widetilde{K}_i = \langle x_{i,j} (r) : (i,j) \in C_{\gamma_0^i},  r \in \mathbb{F}_p [t], \deg (r) \leq 3 \rangle.$$
We note that the natural surjective homomorphism $\widetilde{\Gamma}_{n,p} \rightarrow \SL_{n+1} (\mathbb{F}_p [t])$ maps the every subgroup $\widetilde{K}_i $ bijectively to $K_i$ (where $K_i$ is as in \Cref{The Kaufman-Oppenheim coset complexes subsec}).  Thus, by abuse of notation,  below we will denote $K_i$ in lieu of $\widetilde{K}_i$ (treating $K_i$ as both a subgroup of $ \SL_{n+1} (\mathbb{F}_p [t])$ and of $\widetilde{\Gamma}_{n,p}$).

Let $\widetilde{X} = \mathcal{CC} (\widetilde{\Gamma}_{n,p},  \lbrace K_i \rbrace_{0 \leq i \leq n})$.  We note that all the relations in the presentation of $\widetilde{\Gamma}_{n,p}$ are relations that occur in the subgroups $K_i$.  Thus,  by \cite[Theorem 2.4]{AbelsH},  the complex $\widetilde{X}$ is connected and simply connected and in particular,  for every group $\Lambda$ it holds that $H^1 (\widetilde{X}, \Lambda) =0$.  We note that $\widetilde{X}$ is actually the universal cover of $X$, but we will make no use of this fact.

\begin{theorem}
\label{vanishing of coho for X_n,p thm}
Let $n \geq 3$ and $p$ an odd prime.  Denote
$$X = X_{n,p} = \mathcal{CC} (\SL_{n+1} (\mathbb{F}_p [t] ), \lbrace K_i \rbrace_{0 \leq i \leq n}), $$
where $K_i$ are as in \Cref{The Kaufman-Oppenheim coset complexes subsec}.  Then for every group $\Lambda$ that does not have non-trivial elements of order $p$,  it holds that $H^1 (X, \Lambda) =0$.
\end{theorem}

\begin{proof}
Using the notation of Theorem \ref{propagation thm},  for every pair of roots $\lbrace (i,j), (i',j') \rbrace < \mathcal{C}_0$,  the $\lbrace (i,j), (i',j') \rbrace$-relations holds in $\widetilde{\Gamma}_{n,p}$.  By the definition of $\mathcal{C}_1$ and of pre-$(n,p)$-chamber groups,  it follows that for every $C \in \mathcal{C}_1$,  the subgroup
$$K_C = \langle x_{i,j} (r) : (i,j) \in C,   r \in \mathbb{F}_p [t], \deg (r) \leq 3 \rangle$$
is a pre-$(n,p)$-chamber group.

Thus, by applying Theorem \ref{making pre chamber into chamber and vanishing of H1 thm} for every such group,  we can pass to a quotient $\Gamma '$ in which every $K_C$ subgroup with $C \in \mathcal{C}_1$ is a $(n,p)$-chamber group and for $X' = \mathcal{CC} (\Gamma ' ,   \lbrace K_i \rbrace_{0 \leq i \leq n\rbrace})$,  $H^1 (X' , \Lambda) = 0$.

We repeat the same argument for $C \in \mathcal{C}_2$ and pass to a quotient group $\Gamma ''$ in which every $K_C$ subgroup with $C \in \mathcal{C}_2$ is a $(n,p)$-chamber group and for $X'' = \mathcal{CC} (\Gamma '' ,   \lbrace K_i \rbrace_{0 \leq i \leq n\rbrace})$,  $H^1 (X'' , \Lambda) = 0$.

We note that by Theorem \ref{propagation thm},  for every non-opposite pair of roots $(i,j), (i',j')$,  it holds that $\lbrace (i,j), (i',j') \rbrace < \mathcal{C}_2$ and thus in $\Gamma ''$ for every such pair the $\lbrace (i,j), (i',j') \rbrace$-relations hold.  It follows from Corollary \ref{pres. of SLn roots coro},  that $\Gamma '' =
\SL_{n+1} (\mathbb{F}_p [t])$ and that $X =X''$.  Thus,  $H^1 (X, \Lambda) = 0$ as needed.
\end{proof}

\begin{theorem}
\label{vanishing of coho for X_n,p^s thm}
Let $n \geq 3$,  $s >3n$ and $p$ an odd prime.  Denote
$$X_{n,p}^{(s)} = \mathcal{CC} (\SL_{n+1} (\mathbb{F}_p [t] / \langle t^s \rangle ), \lbrace K_i \rbrace_{0 \leq i \leq n}), $$
where $K_i$ are as in \Cref{The Kaufman-Oppenheim coset complexes subsec}.  Then for every group $\Lambda$ that does not have non-trivial elements of order $p$,  it holds that $H^1 (X_{n,p}^{(s)} , \Lambda) =0$.
\end{theorem}

\begin{proof}
Denote
$$X = X_{n,p} = \mathcal{CC} (\SL_{n+1} (\mathbb{F}_p [t] ), \lbrace K_i \rbrace_{0 \leq i \leq n}).$$
 By Theorem \ref{vanishing of coho for X_n,p thm},  $H^1 (X, \Lambda) =0$.

Denote $\Gamma_{n,p}^s = \Ker (\SL_{n +1} (\mathbb{F}_p [t]) \rightarrow \SL_{n +1} (\mathbb{F}_p [t] / \langle t^s \rangle))$ and note that these are normal subgroups such that for $X_{n,p} = \mathcal{CC} (\SL_{n +1} (\mathbb{F}_p [t]),  \lbrace K_i \rbrace_{0 \leq i \leq n})$ it holds that $X_{n,p}^{(s)} = \Gamma_{n,p}^s \backslash X_{n,p}$.

By Theorem \ref{Stn = SLn in Fp t case thm},  $\SL_{n+1} (\mathbb{F}_p [t]) = \St_{n+1} (\mathbb{F}_p [t])$ and thus,  by Observation \ref{kernel of St (R/I) obs},  $\Gamma_{n,p}^s$ is normally generated by elements of the form $\lbrace e_{i,j} (t^{k}) : 1 \leq i, j \leq n+1,  i\ neq j,  k > s \rbrace$ and every such element is of order $p$.

It follows by Corollary \ref{p coset complex H^1 coro} that $H^1 (X_{n,p}^{(s)} , \Lambda) =0$ as needed.
\end{proof}

\begin{theorem}
Let $n \in \mathbb{N},  n \geq 3$ and $\Lambda$ a finite group.  There is a constant $p_1$ such that for every $p \geq p_1$ prime,  the family $\lbrace X_{n,p}^{(s)} : s >3n \rbrace$ is a family of uniformly bounded degree that has $1$-coboundary expansion over $\Lambda$.
\end{theorem}

\begin{proof}
Fix $n,  \Lambda$ as above.  Let $p_0$ be the constant of Theorem \ref{cosys exp of KO complexes thm} and let $p_1 = \max \lbrace p_0,  \vert \Lambda \vert +1\rbrace$.  For every prime $p \geq p_1$ the following holds: first,  by Theorem \ref{cosys exp of KO complexes thm},  the family $\lbrace X_{n,p}^{(s)} : s >3n \rbrace$ is a family of uniformly bounded degree $\Lambda$ 1-cosystolic expanders.  Second,  since $p > \vert \Lambda \vert$,  the group $\Lambda$ does not contain a non-trivial element of order $p$.  Thus,  by Theorem \ref{vanishing of coho for X_n,p^s thm},  for every $X_{n,p}^{(s)}$,  $H^1 (X_{n,p}^{(s)},   \Lambda) =0$.  It follows that  the family $\lbrace X_{n,p}^{(s)} : s >3n \rbrace$ is a family of uniformly bounded degree $\Lambda$ 1-coboundary expanders.
\end{proof}

\bibliographystyle{alpha}
\bibliography{bibl}

\newcommand{\etalchar}[1]{$^{#1}$}
\begin{thebibliography}{DEL{\etalchar{+}}22}

\bibitem[AH93]{AbelsH}
Herbert Abels and Stephan Holz.
\newblock Higher generation by subgroups.
\newblock {\em J. Algebra}, 160(2):310--341, 1993.

\bibitem[BD01]{BissD}
Daniel~K. Biss and Samit Dasgupta.
\newblock A presentation for the unipotent group over rings with identity.
\newblock {\em J. Algebra}, 237(2):691--707, 2001.

\bibitem[BLM24]{BLM}
Mitali Bafna, Noam Lifshitz, and Dor Minzer.
\newblock Constant degree direct product testers with small soundness.
\newblock \url{https://arxiv.org/abs/2402.00850}, 2024.

\bibitem[BM24]{BM}
Mitali Bafna and Dor Minzer.
\newblock Characterizing direct product testing via coboundary expansion.
\newblock In {\em S{TOC}'24---{P}roceedings of the 56th {A}nnual {ACM}
  {S}ymposium on {T}heory of {C}omputing}, pages 1978--1989. ACM, New York,
  [2024] \copyright 2024.

\bibitem[BMV24]{BMV}
Mitali Bafna, Dor Minzer, and Nikhil Vyas.
\newblock Quasi-linear size pcps with small soundness from hdx.
\newblock \url{https://arxiv.org/abs/2407.12762}, 2024.

\bibitem[CL24]{CL}
Michael Chapman and Alexander Lubotzky.
\newblock Stability of homomorphisms, coverings and cocycles {II}: Examples,
  applications and open problems.
\newblock \url{https://arxiv.org/abs/2311.06706}, 2024.

\bibitem[DD24a]{DD-cosys}
Yotam Dikstein and Irit Dinur.
\newblock {Coboundary and Cosystolic Expansion Without Dependence on Dimension
  or Degree}.
\newblock In Amit Kumar and Noga Ron-Zewi, editors, {\em Approximation,
  Randomization, and Combinatorial Optimization. Algorithms and Techniques
  (APPROX/RANDOM 2024)}, volume 317 of {\em Leibniz International Proceedings
  in Informatics (LIPIcs)}, pages 62:1--62:24, Dagstuhl, Germany, 2024. Schloss
  Dagstuhl -- Leibniz-Zentrum f{\"u}r Informatik.

\bibitem[DD24b]{DD-swap}
Yotam Dikstein and Irit Dinur.
\newblock Swap cosystolic expansion.
\newblock In {\em Proceedings of the 56th Annual ACM Symposium on Theory of
  Computing}, STOC 2024, pages 1956--1966, New York, NY, USA, 2024. Association
  for Computing Machinery.

\bibitem[DDL24]{DDL}
Yotam Dikstein, Irit Dinur, and Alexander Lubotzky.
\newblock Low acceptance agreement tests via bounded-degree symplectic hdxs.
\newblock \url{https://arxiv.org/abs/2402.01078}, 2024.

\bibitem[DEL{\etalchar{+}}22]{DELLM}
Irit Dinur, Shai Evra, Ron Livne, Alexander Lubotzky, and Shahar Mozes.
\newblock Locally testable codes with constant rate, distance, and locality.
\newblock In {\em S{TOC} '22---{P}roceedings of the 54th {A}nnual {ACM}
  {SIGACT} {S}ymposium on {T}heory of {C}omputing}, pages 357--374. ACM, New
  York, [2022] \copyright 2022.

\bibitem[DK17]{DK}
Irit Dinur and Tali Kaufman.
\newblock High dimensional expanders imply agreement expanders.
\newblock In {\em 58th {A}nnual {IEEE} {S}ymposium on {F}oundations of
  {C}omputer {S}cience---{FOCS} 2017}, pages 974--985. IEEE Computer Soc., Los
  Alamitos, CA, 2017.

\bibitem[DLYZ23]{NewCodes}
Irit Dinur, Siqi Liu, and Rachel Yun~Zhang.
\newblock New codes on high dimensional expanders.
\newblock \url{https://arxiv.org/abs/2308.15563}, 2023.

\bibitem[DM22]{DM}
Irit Dinur and Roy Meshulam.
\newblock Near coverings and cosystolic expansion.
\newblock {\em Arch. Math. (Basel)}, 118(5):549--561, 2022.

\bibitem[dPVB24]{KacM-complexes}
Laura~Grave de~Peralta and Inga Valentiner-Branth.
\newblock High-dimensional expanders from kac-moody-steinberg groups.
\newblock \url{https://arxiv.org/abs/2401.05197}, 2024.

\bibitem[EK24]{EK}
Shai Evra and Tali Kaufman.
\newblock Bounded degree cosystolic expanders of every dimension.
\newblock {\em J. Amer. Math. Soc.}, 37(1):39--68, 2024.

\bibitem[EKZ20]{EKZ}
Shai Evra, Tali Kaufman, and Gilles Z\'emor.
\newblock Decodable quantum {LDPC} codes beyond the square root distance
  barrier using high dimensional expanders.
\newblock In {\em 2020 {IEEE} 61st {A}nnual {S}ymposium on {F}oundations of
  {C}omputer {S}cience}, pages 218--227. IEEE Computer Soc., Los Alamitos, CA,
  [2020] \copyright 2020.

\bibitem[FK24]{FK}
Uriya~A. First and Tali Kaufman.
\newblock Cosystolic expansion of sheaves on posets with applications to good
  2-query locally testable codes and lifted codes.
\newblock In {\em S{TOC}'24---{P}roceedings of the 56th {A}nnual {ACM}
  {S}ymposium on {T}heory of {C}omputing}, pages 1446--1457. ACM, New York,
  [2024] \copyright 2024.

\bibitem[GK23]{GK}
Roy Gotlib and Tali Kaufman.
\newblock {List Agreement Expansion from Coboundary Expansion}.
\newblock In Yael Tauman~Kalai, editor, {\em 14th Innovations in Theoretical
  Computer Science Conference (ITCS 2023)}, volume 251 of {\em Leibniz
  International Proceedings in Informatics (LIPIcs)}, pages 61:1--61:23,
  Dagstuhl, Germany, 2023. Schloss Dagstuhl -- Leibniz-Zentrum f{\"u}r
  Informatik.

\bibitem[Gro10]{Grom}
Mikhail Gromov.
\newblock Singularities, expanders and topology of maps. {P}art 2: {F}rom
  combinatorics to topology via algebraic isoperimetry.
\newblock {\em Geom. Funct. Anal.}, 20(2):416--526, 2010.

\bibitem[KKL14]{KKL}
Tali Kaufman, David Kazhdan, and Alexander Lubotzky.
\newblock Ramanujan complexes and bounded degree topological expanders.
\newblock In {\em 55th {A}nnual {IEEE} {S}ymposium on {F}oundations of
  {C}omputer {S}cience---{FOCS} 2014}, pages 484--493. IEEE Computer Soc., Los
  Alamitos, CA, 2014.

\bibitem[KM21]{KM}
Tali Kaufman and David Mass.
\newblock {Unique-Neighbor-Like Expansion and Group-Independent Cosystolic
  Expansion}.
\newblock In Hee-Kap Ahn and Kunihiko Sadakane, editors, {\em 32nd
  International Symposium on Algorithms and Computation (ISAAC 2021)}, volume
  212 of {\em Leibniz International Proceedings in Informatics (LIPIcs)}, pages
  56:1--56:17, Dagstuhl, Germany, 2021. Schloss Dagstuhl -- Leibniz-Zentrum
  f{\"u}r Informatik.

\bibitem[KO18]{KO-construction}
Tali Kaufman and Izhar Oppenheim.
\newblock Construction of new local spectral high dimensional expanders.
\newblock In {\em S{TOC}'18---{P}roceedings of the 50th {A}nnual {ACM} {SIGACT}
  {S}ymposium on {T}heory of {C}omputing}, pages 773--786. ACM, New York, 2018.

\bibitem[KO21]{KO-cobound}
Tali Kaufman and Izhar Oppenheim.
\newblock Coboundary and cosystolic expansion from strong symmetry.
\newblock In {\em 48th {I}nternational {C}olloquium on {A}utomata, {L}anguages,
  and {P}rogramming}, volume 198 of {\em LIPIcs. Leibniz Int. Proc. Inform.},
  pages Art. No. 84, 16. Schloss Dagstuhl. Leibniz-Zent. Inform., Wadern, 2021.

\bibitem[KO23]{KO-construction2}
Tali Kaufman and Izhar Oppenheim.
\newblock High dimensional expanders and coset geometries.
\newblock {\em European J. Combin.}, 111:Paper No. 103696, 31, 2023.
\newblock With a preface by Alexander Lubotzky.

\bibitem[LM06]{LM}
Nathan Linial and Roy Meshulam.
\newblock Homological connectivity of random 2-complexes.
\newblock {\em Combinatorica}, 26(4):475--487, 2006.

\bibitem[LSV05a]{LSV2}
Alexander Lubotzky, Beth Samuels, and Uzi Vishne.
\newblock Explicit constructions of {R}amanujan complexes of type {$\tilde
  A_d$}.
\newblock {\em European J. Combin.}, 26(6):965--993, 2005.

\bibitem[LSV05b]{LSV1}
Alexander Lubotzky, Beth Samuels, and Uzi Vishne.
\newblock Ramanujan complexes of type {$\tilde A_d$}.
\newblock volume 149, pages 267--299. 2005.
\newblock Probability in mathematics.

\bibitem[OP22]{O'DP}
Ryan O'Donnell and Kevin Pratt.
\newblock High-dimensional expanders from {C}hevalley groups.
\newblock In {\em 37th {C}omputational {C}omplexity {C}onference}, volume 234
  of {\em LIPIcs. Leibniz Int. Proc. Inform.}, pages Art. No. 18, 26. Schloss
  Dagstuhl. Leibniz-Zent. Inform., Wadern, 2022.

\bibitem[Opp18]{OppLocI}
Izhar Oppenheim.
\newblock Local spectral expansion approach to high dimensional expanders
  {P}art {I}: {D}escent of spectral gaps.
\newblock {\em Discrete Comput. Geom.}, 59(2):293--330, 2018.

\bibitem[PK22]{PK}
Pavel Panteleev and Gleb Kalachev.
\newblock Asymptotically good quantum and locally testable classical {LDPC}
  codes.
\newblock In {\em S{TOC} '22---{P}roceedings of the 54th {A}nnual {ACM}
  {SIGACT} {S}ymposium on {T}heory of {C}omputing}, pages 375--388. ACM, New
  York, [2022] \copyright 2022.

\bibitem[Spl86]{Split}
Siegfried Splitthoff.
\newblock Finite presentability of {S}teinberg groups and related {C}hevalley
  groups.
\newblock In {\em Applications of algebraic {$K$}-theory to algebraic geometry
  and number theory, {P}art {I}, {II} ({B}oulder, {C}olo., 1983)}, volume~55 of
  {\em Contemp. Math.}, pages 635--687. Amer. Math. Soc., Providence, RI, 1986.

\end{thebibliography}
\end{document}